\documentclass[11pt,oneside,english]{amsart}
\usepackage[utf8]{inputenc}
\usepackage{amsfonts, amsmath, amssymb, amsthm, mathrsfs}
\usepackage{biblatex}
\usepackage{dsfont}
\usepackage{graphicx}
\usepackage{textcomp}
\usepackage{etoolbox}
\usepackage{eso-pic}
\usepackage{verbatim}
\usepackage{comment}
\usepackage{blindtext}
\usepackage{enumitem}
\usepackage{thmtools}
\usepackage{thm-restate}

\usepackage{lmodern}
\usepackage{geometry}
\usepackage[colorinlistoftodos]{todonotes}

\usepackage{tikz}
\usetikzlibrary{3d, patterns, calc}
\usepackage{mathtools}
\usepackage{xcolor}
\usepackage{marginnote}
\newcommand{\bb}[1]{\mathbb{#1}}

\newcommand{\const}[1]{k_{#1}}
\renewcommand{\epsilon}{\varepsilon}
\newcommand{\llb}{\left\lbrace}

\newcommand{\rrb}{\right\rbrace}
\newcommand{\vass}[1]{\left| #1\right|}

\newcommand{\Vbin}{V_{\mathrm{bin}}}
\newcommand{\eps}{\varepsilon}

\theoremstyle{plain}
\newtheorem{theorem}{Theorem}[section]
\newtheorem{lemma}[theorem]{Lemma} 
\newtheorem*{lemma*}{Lemma} 
\newtheorem{corollary}[theorem]{Corollary}

\newtheorem{setting}{Setting}

\theoremstyle{remark}

\newtheorem{claim}[theorem]{Claim}

\theoremstyle{definition}
\newtheorem{definition}[theorem]{Definition}

\newcommand{\oldqed}{}
\newcommand{\qedClaim}{\hfill\scalebox{.6}{$\Box$}} 
\newenvironment{claimproof}[1][Proof]{
	\renewcommand{\oldqed}{\qedsymbol}
	\renewcommand{\qedsymbol}{\qedClaim}
	\begin{proof}[#1]
	}{
	\end{proof}
	\renewcommand{\qedsymbol}{\oldqed}
}

\usepackage{hyperref}

\renewbibmacro{in:}{}
\DeclareFieldFormat[article,periodical]{volume}{\mkbibbold{#1}}
\DeclareFieldFormat
  [article, online, inproceedings, unpublished, misc, inbook]
  {journaltitle}{#1}
\DeclareFieldFormat
  [inproceedings]
  {booktitle}{#1}
\DeclareFieldFormat
  [article,online,inproceedings, unpublished, misc, inbook]
  {title}{\textit{#1}}

  \title{The Ramsey numbers of squares of paths and cycles}
  \author{Peter Allen, Domenico Mergoni Cecchelli, Barnaby Roberts and Jozef Skokan}
  \address[P.~Allen $\vert$ D. Mergoni $\vert$ B.~Roberts $\vert$ J.~Skokan]{London School of Economics, London, WC2A 2AE, UK.}
\email{(p.d.allen$\vert$d.mergoni$\vert$j.skokan)@lse.ac.uk}
\email{roberts.barnaby@gmail.com}
\address[J.~Skokan]{Department of Mathematics, University of Illinois at Urbana-Champaign,  Urbana, IL 61801, USA.}
  \date{}

\begin{document}
\maketitle
\begin{abstract}
The square $G^2$ of a graph $G$ is the graph on $V(G)$ with a pair of vertices $uv$ an edge whenever $u$ and $v$ have distance $1$ or $2$ in $G$. Given graphs $G$ and $H$, the Ramsey number $R(G,H)$ is the minimum $N$ such that whenever the edges of the complete graph $K_N$ are coloured with red and blue, there exists either a red copy of $G$ or a blue copy~of~$H$.

We prove that for all sufficiently large $n$ we have
\[R(P_{3n}^2,P_{3n}^2)=R(P_{3n+1}^2,P_{3n+1}^2)=R(C_{3n}^2,C_{3n}^2)=9n-3\mbox{ and } R(P_{3n+2}^2,P_{3n+2}^2)=9n+1.\] 

We also show that for any $\gamma>0$ and $\Delta$ there exists $\beta>0$ such that the following holds: If $G$ can be coloured with three colours such that all colour classes have size at most $n$, the maximum degree $\Delta(G)$ of $G$ is at most $\Delta$, and $G$ has bandwidth at most $\beta n$, then $R(G,G)\le (3+\gamma)n$.
\end{abstract}

\thispagestyle{empty}
\section{Introduction}
Given graphs $G$ and $H$, the Ramsey number $R(G,H)$ is the minimum $N$ such that whenever the edges of the complete graph on $N$ vertices $K_N$ are coloured with red and blue, there exists either a red copy of $G$ or a blue copy of $H$.

The study of Ramsey numbers has a long history, and in general it is hard to find even good upper and lower bounds on $R(G,H)$. In this paper, we are interested in the case that $G$ and $H$ are sparse graphs. In this case, if $G$ is connected and $v(G)\ge\sigma(H)$, one has the lower bound 
\begin{equation}\label{eq:burr}
R(G,H)\ge \big(\chi(H)-1\big)\big(v(G)-1\big)+\sigma(H).
\end{equation}
Here $v(G)$ denotes the number of vertices of $G$, $\chi(H)$ is the chromatic number of $H$, and $\sigma(H)$ is the minimum, over all $\chi(H)$-colourings of $H$, of the smallest colour class size. This lower bound is due to Burr~\cite{Burr}, with the corresponding construction being $\chi(H)-1$ vertex-disjoint red cliques each on $v(G)-1$ vertices, plus one further red clique on $\sigma(H)-1$ vertices, and all other edges blue. When this construction gives the Ramsey number (i.e. when we have an equality in \eqref{eq:burr}), we say that  $G$ is $H$\emph{-good}.

For fixed graphs $H$, the class of graphs $G$ which are $H$-good is quite well understood; see Allen, Brightwell and Skokan~\cite{ABS} and Nikiforov and Rousseau~\cite{NikiRous}. However much less is known about the case when $H$ grows with $v(G)$, or when $H=G$. Burr~\cite{Burr} conjectured that for fixed $\Delta$, every connected graph $G$ with $\Delta(G)\le\Delta$ and $v(G)$ large enough is $G$-good. This statement holds for $G=P_n$~~\cite{GG67} and $G=C_n$~\cite{BE73,R73}. However it was disproved by Graham, R\"odl and Ruci\'nski~\cite{GRR}, who showed that it fails badly for expander graphs, and again in~\cite{ABS}, where a lower bound on $R(P_n^k,P_n^k)$ better than \eqref{eq:burr} is shown for each $k\ge2$. In the latter paper, however, it is shown that Burr's conjecture is off by at most a factor (roughly) $2$ when $G$ has bounded maximum degree and sublinear bandwidth. Here the bandwidth of $G$ is the smallest $k$ such that $G$ is a subgraph of $P^k_{v(G)}$. 

In~\cite{ABS}, a value for the Ramsey numbers of squares of paths, and squares of cycles on a number of vertices divisible by $3$, is conjectured. We observe that the conjectured value is wrong by one, and prove the modified conjecture.

\begin{restatable}{theorem}{thmmain}\label{thm:main}
 There exists $n_0$ such that for all $n\ge n_0$ we have:
 \[R(P_{3n}^2,P_{3n}^2)=R(P_{3n+1}^2,P_{3n+1}^2)=R(C_{3n}^2,C_{3n}^2)=9n-3\mbox{ and } R(P_{3n+2}^2,P_{3n+2}^2)=9n+1.\]
\end{restatable}
The lower bound part of this theorem is the following construction from~\cite{ABS}. We take disjoint vertex sets $X_1,X_2,Y_1,Y_2$ each with $2n-1$ vertices, plus $Z$ with $n-1$ vertices. We colour edges within each $X_i$ blue and within each $Y_i$ red. We colour edges in the bipartite graphs $(X_1,X_2)$ and $(X_i,Z)$ red, and in $(Y_1,Y_2)$ and $(Y_i,Z)$ blue. We colour $(X_1,X_2)$ and $(Y_1,Y_2)$ blue, and $(X_1,Y_2)$ and $(X_2,Y_1)$ red. Finally, we add a single vertex $z$, which sends blue edges to $X_1\cup X_2$ and red to $Y_1\cup Y_2$. The edges within $Z\cup\{z\}$ may be coloured arbitrarily, as illustrated in Figure~\ref{fig:construct}. A short case analysis demonstrates that this construction does not contain a monochromatic $P_{3n}^2$. Furthermore, we can add one extra vertex to each of $X_1,X_2,Y_1,Y_2$ and still have no $P_{3n+2}^2$.

\begin{figure}[h]
\begin{center}
\begin{tikzpicture}
\node (A) at ( 2,7) [circle, fill=blue!80,inner sep=8pt,label=left:$X_1$] {$2n-1$};
\node (B) at ( 2,4) [circle, fill=blue!80,inner sep=8pt,label=left:$X_2$] {$2n-1$};
\node (C) at ( 4,2) [circle, fill=purple!80,inner sep=5pt,label=left:$Z$] {$n-1$};
\node (D) at ( 6,7) [circle, fill=red,inner sep=8pt,label=right:$Y_1$] {$2n-1$};
\node (E) at ( 6,4) [circle, fill=red,inner sep=8pt,label=right:$Y_2$] {$2n-1$};
\node (z) at ( 4,9) [circle, fill=red,inner sep=2pt,label=left:$z$] {};

  \foreach \from/\to in {A/B,A/C,B/C,A/E,B/D}
    \draw [red,line width=5pt] (\from) -- (\to);
   
  \foreach \from/\to in {D/E,D/C,E/C,A/D,B/E}
    \draw [blue!80,line width=5pt] (\from) -- (\to); 
    
  \foreach \from/\to in {z/A,z/B}
    \draw [blue!80,line width=2pt] (\from) -- (\to);
   
  \foreach \from/\to in {z/D,z/E}
    \draw [red,line width=2pt] (\from) -- (\to); 
    
  \foreach \from/\to in {z/C}
    \draw [purple!30,line width=2pt] (\from) -- (\to); 
\end{tikzpicture}
\end{center}
\caption{Lower bound construction}\label{fig:construct}
\end{figure}
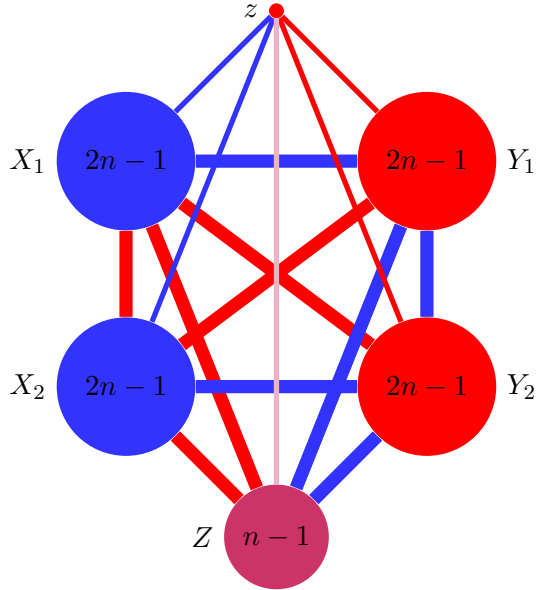

In addition, we give a general upper bound on Ramsey numbers for $3$-colourable graphs with bounded maximum degree and sublinear bandwidth, which $P_{3n}^2$ demonstrates is asymptotically tight.

\begin{theorem}\label{thm:boundBW}
 Given $\gamma>0$ and $\Delta$, there exist $\beta>0$ and $n_0$ such that for all $n\ge n_0$ the following holds. Suppose that $H$ is a graph with $\Delta(H)\le\Delta$, with bandwidth at most $\beta n$, and with a proper vertex $3$-colouring all of whose colour classes have at most $n$ vertices. Then $R(H,H)\le(9+\gamma)n$.
\end{theorem}

We recall from~\cite{ABS} that the bandwidth restriction in this theorem is necessary: for any given $\beta>0$, if $\Delta$ is large enough there are $n$-vertex graphs $H$ with bandwidth at most $\beta n$ and maximum degree at most $\Delta$ for which the theorem statement is false.

\medskip

Our proof method uses the stability-extremal paradigm. Using the Szemer\'edi Regularity Lemma and the Blow-up Lemma, we will argue that to find a monochromatic square of a path (or cycle, or $3$-colourable sparse graph as in Theorem~\ref{thm:boundBW}) it is enough to find in the cluster graph a monochromatic triangle factor which is `triangle connected' (which we will define later). This standard reduction leaves us looking, in a nearly complete edge-coloured graph, for a large monochromatic triangle-connected triangle factor (TCTF). The main technical work of the paper (Lemma~\ref{MainLemma}) is then to prove that a $2$-edge-coloured near complete graph on nearly $9t$ vertices will either contain a monochromatic TCTF on a little more than $3t$ vertices, or alternatively the graph must be close to the extremal example.

To prove the main lemma, we use a second partitioning method, as in~\cite{ABS}: by an iterative use of Ramsey's theorem, we partition most of the $9t$ vertices into a collection of bounded size (but quite large) monochromatic cliques. Obviously, it is easy to find a large red triangle factor in a collection of red cliques: in addition, we will see that two triangles in (or even using one edge of) the same red clique are `triangle connected' in red, and that if two red cliques are \emph{not} red triangle connected, then almost all the edges between them have to be blue. These observations were previously made in~\cite{ABS}. Where we improve compared to that paper is that we are able to deal with the interaction between cliques of different colours (whereas in~\cite{ABS} the minority colour cliques are thrown away).

\section{Notation, main lemmas and organisation}\label{sec:note}
Our graph notation is mainly standard. We will often write $|G|$ for the number of vertices in a graph $G$, and similarly $|M|$ for the number of vertices covered by a matching $M$ (i.e.\ twice the number of edges of $M$). We will often want to refer to edges (of a given colour) between two or three vertex sets. We write $(A,B)$ or $(A,B,C)$ for respectively $\{ab:a\in A,b\in B\}$ and $(A,B)\cup(A,C)\cup(B,C)$, the graph we refer to will always be clear from the context. We will work with $2$-edge-coloured graphs, and refer to the two colours as `red' and `blue'.

Given a graph $G$, we say that edges $uv$ and $uw$ of $G$ are \emph{triangle-connected} if $vw$ is an edge of $G$, we extend this to an equivalence relation on edges by transitive closure. We refer to the equivalence classes of this relation as \emph{triangle components}. We will generally want to talk about monochromatic triangle connection. Thus, if the edges of $G$ are $2$-coloured, we say that two red edges are red triangle connected if they are triangle connected in the subgraph of $G$ consisting only of red edges, we define red triangle component similarly. We also, slightly abusing notation, will say two red cliques (each with at least two vertices) are red triangle-connected if an edge (and so all edges) in one is red triangle connected to an edge (so all edges) of the other. When the colour is clear from the context (as with \emph{red} cliques) we will often just say that the two cliques are triangle connected.

A \emph{triangle factor} in a graph $G$ is a collection of vertex-disjoint triangles of $G$. It is a \emph{triangle-connected triangle factor} (TCTF) if all its edges lie in a single triangle component. Again, we will usually want to talk about monochromatic TCTFs in a $2$-edge-coloured graph $G$, and as above a red TCTF means a TCTF in the subgraph of red edges of $G$.

At this point, we are in a position to give the case analysis proving the lower bound part of Theorem~\ref{thm:main}.
\begin{proof}[Proof of Theorem~\ref{thm:main}, lower bounds]
 We begin by describing the red triangle components of the lower bound construction for $P_{3n}^2$, $P_{3n+1}^2$ and $C_{3n}^2$. The edges in $Y_1$ and in $(Y_1,X_2\cup\{z\})$, form a red triangle component. Similarly, the edges in $Y_2$ and $(Y_2,X_1\cup\{z\})$ form a red triangle component. The edges $(X_1,X_2,Z)$, together with all red edges in $Z$ and all red edges from $z$ to $Z$ which lie in a red triangle, form a red triangle component. Finally, each red edge from $z$ to $Z$ which is not in a red triangle forms a triangle component. The blue components are analogous.
 
 If the lower bound construction contains a red $P_{3n}^2$, then in particular it has a red triangle component which contains a red triangle factor with $n$ triangles. Checking each entry in the list above, observe that removing $Y_1$ from the first leaves an independent set: $X_2\cup\{z\}$ contains no red edges. But $Y_1$ contains only $2n-1$ vertices, so there cannot be a $3n$-vertex triangle factor in this component. The symmetric argument deals with the symmetric second red triangle component. For the third case, removing $Z$ leaves a bipartite graph: the only red edges are those in $(X_1,X_2)$. But $Z$ contains only $n-1$ vertices, so this component too contains no $3n$-vertex red triangle factor. Finally, trivially the single-edge components contain no red triangle factor. The argument to exclude a blue $P_{3n}^2$ is symmetric.
 
 For the modification for $P_{3n+2}^2$, adding one vertex to each of $X_1,X_2,Y_1,Y_2$, the description of triangle components above, and the explanation that the red triangle component containing $(X_1,X_2,Z)$ does not contain $P_{3n}^2$ continues to work. Observe that $P_{3n+2}^2$ has independence number $n+1$, so removing any $2n$ vertices leaves at least one edge. This observation shows that the red component consisting of edges in $Y_1$ and $(Y_1,X_2\cup\{z\})$ does not contain a red $P_{3n+2}^2$, and the other cases are symmetric.
\end{proof}

\medskip

The main work of this paper is to prove the following stability lemma, which states that a $2$-edge-coloured nearly complete graph $G$ on almost $9t$ vertices either contains a monochromatic TCTF on a little more than $3t$ vertices, or is close to the extremal example. To state it, we need one further definition.

Given an edge-coloured graph $G$, let $A\subseteq V(G)$ and $v$ a vertex of $G$ not in $A$. For $r\in\bb{R}$, we say that $v$ is \emph{$r$-blue to $A$} if $va$ is a blue edge of $G$ for all but at most $r$ vertices  $a\in A$. Similarly, given $A,B\subseteq V(G)$ disjoint, we say that \emph{$(A,B)$ is $r$-blue} if all but at most $r$ vertices in $A$ are $r$-blue to $B$ and vice versa. We define similarly \emph{$r$-red}.

We will generally use this notation with $r$ much smaller than the sets $A$ and $B$, so the reader can think of $r$-blue as meaning `almost all blue'. Our main lemma is then the following.

\begin{restatable}{lemma}{MainLem}\label{MainLemma}
There exists $\delta_0>0$ such that for any $0<h,\lambda<\delta_0$ there exist $\epsilon_0>0$ and $t_0\in\bb{N}$ such that for any $t\geq t_0$ and $0<\epsilon<\epsilon_0$ the following holds. Let $G$ be a $2$-edge-coloured graph on $(9-\epsilon)t$ vertices with minimum degree at least $(9-2\epsilon)t$. Then either $G$ contains a monochromatic TCTF on at least $3(1+\epsilon)t$ vertices or $V(G)$ can be partitioned in sets $B_1, B_2, R_1, R_2, Z, T$ such that the following hold.
\begin{enumerate}[label=(\alph*)]
    \item $(2-h)t\leq \vass{B_1}, \vass{B_2}, \vass{R_1}, \vass{R_2}\leq (2+h)t$,
    \item $(1-h)t\leq \vass{Z}\leq (1+h)t$,
    \item all the edges in $G[B_1]$ and $G[B_2]$ are blue, and all the edges in $G[R_1]$ and $G[R_2]$ are red,
    \item all the edges between the pairs $(B_1,R_1)$, $(B_2, R_2)$, $(R_1, Z)$ and $(R_2,Z)$ are blue, and those between the pairs $(B_1,R_2)$, $(B_2, R_1)$, $(B_1, Z)$ and $(B_2,Z)$ are red,
    \item the pair $(B_1, B_2)$ is $\lambda t$-red, and the pair $(R_1, R_2)$ is $\lambda t$-blue, and
    \item $\vass{T}\leq ht$.
\end{enumerate}
\end{restatable}
We will prove this lemma in Sections~\ref{sec:genset}--\ref{sec:finproof}.

By applying the Regularity Method in a standard way, we are able to upgrade Lemma~\ref{MainLemma} to the following superficially similar statement, in which we replace TCTF with the square of a path and cycle. We could generalise the following lemma to nearly-complete graphs easily (as in Lemma~\ref{MainLemma}), but we do not need it for the proof.

\begin{restatable}{lemma}{RegResult}\label{regularityresult} For every $\alpha>0$ there exists $\delta>0$ and $n_0\in\mathbb{N}$ such that for every $n>n_0$ the following holds. Let $N\geq (9-\delta)n$, and let $G$ be a $2$-edge-colouring of $K_N$. Then either $G$ contains both a monochromatic copy of $P_{3n+2}^2$ and of $C_{3n}^2$, or we can partition $V(G)$ into sets $X_1, X_2, Y_1, Y_2, Z$ and $R$ such that the following hold.
\begin{enumerate}[label=(\alph*)]
 \item\label{conditionA} $(2-\alpha)n\leq \vass{X_1}, \vass{X_2},\vass{Y_1},\vass{Y_2}\leq (2+\alpha)n$,
 \item\label{conditionB} $(1-\alpha)n\leq \vass{Z}\leq (1+\alpha)n$,
 \item\label{conditionC} $\vass{R}\leq \alpha n$,
 \item Vertices in the following pairs have at most $\alpha n$ red neighbours in the opposite part: $(X_1, Y_1), (X_2, Y_2), (Y_1, Y_2), (Y_1, Z)$ and $(Y_2, Z)$,
 \item Vertices in the following pairs have at most $\alpha n$ blue neighbours in the opposite part: $(X_1, X_2), (X_2, Y_1), (X_1, X_2), (X_1, Z)$ and $(X_2, Z)$,
 \item Vertices in $X_1$ and $X_2$ have at most $\alpha n$ red neighbours in their own part,
 \item\label{conditionG} Vertices in $Y_1$ and $Y_2$ have at most $\alpha n$ blue neighbours in their own part.
\end{enumerate}
\end{restatable}
We deduce this lemma from Lemma~\ref{MainLemma} in Section~\ref{sec:reg}.

To complete the proof of Theorem~\ref{thm:main}, we need to show that a complete graph which can be partitioned as in the above Lemma~\ref{regularityresult} and which has $9n-3$ vertices necessarily contains both a monochromatic $P_{3n+1}^2$ and $C_{3n}^2$; and $9n+1$ vertices suffices for $P_{3n+2}^2$. We do this in Section~\ref{sec:exact}.

Finally, to prove Theorem~\ref{thm:boundBW} it suffices to observe that if $G$ satisfies the conditions of Lemma~\ref{MainLemma} and can be partitioned as in that lemma, then it contains a monochromatic TCTF on nearly $3t$ vertices. Together with a standard application of the Regularity Method, which we sketch in Section~\ref{sec:reg}, this completes the proof of Theorem~\ref{thm:boundBW}.

\section{Preliminary lemmas}
In this section we prove some basic Ramsey-theoretic results which we will need to prove Lemma~\ref{MainLemma}, but for which we do \emph{not} assume the conditions of Lemma~\ref{MainLemma}.

\begin{lemma}\label{lemmasteplowerbound1}
There exist $\epsilon_0>0$ and $t\in \bb{R}$ such that the following holds for any $0<\epsilon<\epsilon_0$ and $t>t_0$. Let $G$ be a graph on at least $2(1+3\epsilon)t$ vertices with minimum degree at least $\vass{G}-\epsilon t$. Any 2-edge-colouring of the edges of $G$ contains a red matching on $2(1+\epsilon)t$ vertices or a blue connected matching on $\min\llb \vass{G}-(1+2\epsilon)t, 2\vass{G}-4(1+2\epsilon)t \rrb$ vertices.
\end{lemma}
\begin{proof}
Let $M$ be the largest red matching in $G$ and let $Y=V(G)\setminus M$. We may assume that $M$ has at most $2(1+\epsilon)t$ vertices. Since $M$ is maximal, every edge in $M$ has one endpoint with at most one red neighbour in $Y$. Indeed, if $xy\in M$ and both $x$ and $y$ have at least two neighbours in $Y$ we can take $x'$ in $Y$ adjacent to $x$ and $y'$ distinct from $x'$ adjacent to $y$ in $Y$, and obtain a red matching which is larger than $M$ by substituting $xy$ with $x'x$ and $y'y$.

Let $S$ be the set of vertices in $M$ with at most one red neighbour in $Y$. We can now form a blue matching $P$ (that we are going to show is connected) by greedily matching vertices in $S$ with blue neighbours in $Y$. We claim that $P$ has at least $\min\llb \vass{S}, \vass{G}-\vass{M}-2\epsilon t \rrb$ edges. Indeed, since the process is greedy we stop only by finishing all the vertices of $S$ or when $S\setminus P$ is not empty, but no vertex in $S\setminus P$ has a blue neighbour in $Y\setminus P$, and this means that there are less than $2\epsilon t$ vertices not yet covered by $P$ in $Y$. 

If we stopped for the first reason (if $\vass{S}<\vass{G}-\vass{M}-2\epsilon t$) we can extend $P$ to a larger blue matching $P'$: the induced graph over $Y$ contains only blue edges by maximality of $M$ and there are some edges left in $Y\setminus P$. This extension of $P$ can continue at least until all but $\epsilon t$ vertices in $Y$ are covered: we stop only when all edges in $Y$ have one vertex covered by $P'$. Therefore we have
\begin{align*}
  \vass{V(P')}&\geq \overbrace{2\vass{S}}^{\text{in }P}+\overbrace{\vass{Y}-\vass{S}-\epsilon t}^{\text{in }Y}\\
  &\overbrace{\geq}^{2\vass{S}\geq \vass{M}} \vass{G}-\frac{\vass{M}}{2}-\epsilon t\\
  &\geq \vass{G}-(1+2\epsilon)t\,,
\end{align*}
as desired.

If on the other hand we stopped because no vertex in $S\setminus P$ has a blue neighbour in $Y\setminus P$ (but $S\setminus P$ is not empty). In particular, by definition of $S$ this means that every vertex in $S\setminus P$ has at most one neighbour in $Y\setminus P$. This can only happen if $\vass{Y \setminus P}<2\epsilon t$ and hence all but at most $2\epsilon t$ vertices of $Y$ are covered by $P$. This means that the size of $P$ is at least
\begin{align*}
    \vass{P}&\geq 2 (\vass{Y}-2\epsilon t)\\
    &\geq 2(\vass{G}-\vass{M}-2\epsilon t)\\
    &\geq 2(\vass{G}-2(1+\epsilon)t-2\epsilon t)\\
    &=2\vass{G}-4(1+2\epsilon)t\,,
\end{align*}
as desired.

In order to conclude, we must now argue that the matching $P$ (or $P'$) we obtained is blue connected. But this is the case, indeed, every edge of $P$ (or $P'$) has at least one vertex in $Y$. Indeed $\vass{Y}=\vass{G}-\vass{M}\geq 4\epsilon t$ and all edges in $Y$ are blue. By the minimum degree of $G$ each vertex of $Y$ is non-adjacent to at most $\epsilon t$ vertices of $Y$, so any pair of vertices of $Y$ has a common neighbour in $Y$, and therefore $Y$ is blue-connected.
\end{proof}

\begin{lemma}\label{lemmatwothirdsappe}
Let $G$ be a graph with minimum degree strictly greater than $\frac{2}{3}\vass{G}$. Then all the edges of $G$ are triangle connected. Moreover, there exists a TCTF on all but at most $2$ vertices of $G$.
\end{lemma}
\begin{proof}
We may notice that every three vertices of $G$ share a common neighbour by the minimum degree condition and the pigeonhole principle. This means that any couple of adjacent edges is triangle connected in a trivial way, and this property implies that connected components and triangle-connected components coincide in $G$ (because of the minimum degree condition we have that $G$ is connected and therefore every couple of edges is triangle connected).
The existence of the TCTF is given by a theorem of Corradi and Hajnal \cite{CorradiHajnal1963}.
\end{proof}

\begin{lemma}\label{lemmaconnmatchtctf}
There exist $\epsilon_0>0$ and $t\in\bb{R}$ such that the following holds for any $0<\epsilon<\epsilon_0$, any $t>0$. Let $G$ be a graph on at least $(5+100\epsilon)t$ vertices with minimum degree at least $\vass{G}-\epsilon t$. Any 2-edge-colouring of the edges of $G$ contains a red connected matching over $2(1+\epsilon)t$ vertices or a blue TCTF on $3(1+\epsilon)t$ vertices.
\end{lemma}
\begin{proof} Without loss of generality, we may assume $G$ has $(5+100\epsilon)t$ vertices.
We separate cases.

    \underline{Case 1}: $G$ has a maximal red connected component $A$ that spans at least $(4+5\epsilon)t$ vertices.
    
    Let $M$ be the largest red matching in $A$. Since $A$ is a red connected component, we may assume $\vass{M}<2(1+\epsilon)t$. Since $M$ is a maximal red matching in $A$, we know that every edge in $A\setminus M$ is blue. 
    
    Because of our assumption on the size of $A$, we have that $\vass{A\setminus M}>(2+3\epsilon)t$. We construct a matching $P$ of size $2(1+\epsilon)t$ in $A\setminus M$ greedily, which is possible by the minimum degree of $G$. By Lemma \ref{lemmatwothirdsappe}, every pair of edges in $A\setminus M$ is blue triangle connected. In particular $P$ is blue triangle connected.
    
    We now greedily extend the edges of $P$ to blue triangles by taking vertices in $X=V(G)\setminus (P\cup M)$. We have no red edges from vertices of $X$ to vertices of $P$: if $x\in X$ is not in $A$, this is since $A$ is a red component, while if $x\in X\cap A$ then it is by maximality of $M$. We have $\vass{X}\ge (1+(k-4)\epsilon)t$, and by the minimum degree of $G$ any edge of $P$ makes a triangle with all but at most $2\epsilon t$ vertices of $X$, so the greedy extension succeeds.

\smallskip    
    
\underline{Case 2}: $G$ has a maximal red connected component $A$ that spans at least $3(1+2\epsilon)t$ but less than $(4+5\epsilon)t$ vertices. 

If $G$ has a red connected matching over $2(1+\epsilon)t$ vertices we are done, so we assume it does not. By Lemma \ref{lemmasteplowerbound1} applied to $A$, we obtain a blue connected matching $P$ in $A$ of size at least $2(1+\epsilon)t$. Now as in the previous case, we can greedily extend all the edges of $P$ to a blue triangle factor using vertices of $V(G)\setminus A$. Observe that every two blue adjacent edges in $A$ share a neighbour in $V(G)\setminus A$, therefore every blue connected component in $A$ is also blue triangle connected. In particular, $P$, and hence the blue triangle factor containing it, are triangle connected.

\smallskip    
    
\underline{Case 3}: $G$ has two maximal red connected components $A_1$ and $A_2$ covering at least $(5+12\epsilon)t$ vertices in total, and we are not in Cases 1 or 2.

Because we are not in Cases 1 or 2, $A_1$ and $A_2$ both span less than $3(1+2\epsilon)t$ vertices and hence at least $(2+6\epsilon)t$ vertices. In addition, neither component contains a red matching on $2(1+\epsilon)t$ vertices, because otherwise we would be done. Therefore, each $A_i$ contains a blue connected matching $P_i$ on precisely $\min\big(2\vass{A_i}-4(1+2\epsilon)t,2t\big)$ vertices by Lemma~\ref{lemmasteplowerbound1}. Indeed, for the possible values of $\vass{A_i}$, we have $2\vass{A_i}-4(1+2\epsilon)t<\vass{A_i}-(1+2\epsilon)t$. 
Observe that every edge between $A_1$ and $A_2$ is blue and therefore $P_1\cup P_2$ is a blue connected matching. We have $\vass{P_1},\vass{P_2}\ge 4\epsilon t$ and hence if $\vass{P_1}=2t$ we see that $P_1\cup P_2$ has at least $(1+2\epsilon)t$ edges. Similarly if $\vass{P_2}=2t$. If $\vass{P_1},\vass{P_2}<2 t$ then we have  at least $\vass{A_1}+\vass{A_2}-4(1+2\epsilon)t\geq (1+4\epsilon)t$ edges, in any case we have in $P_1\cup P_2$ at least $(1+2\epsilon)t$ edges. Let $Y_i=A_i\setminus P_i$.
We extend greedily the edges of $P_1$ to a set of disjoint blue triangles $T_1$ using vertices of $Y_2$, and in the same way we greedily extend the edges of $P_2$ to a set of disjoint blue triangles $T_2$ using vertices of $Y_1$. Note that $\vass{Y_i}=4(1+2\epsilon)t-\vass{A_i}> (1+2\epsilon)t$, and therefore we are able to extend the edges of $P_1\cup P_2$, so we obtain a blue triangle factor with at least $(1+\epsilon)t$ triangles.

It now suffices to show that the triangle factor $T_1\cup T_2$ is triangle connected. Because every two blue incident edges in $A_1$ share a neighbour in $A_2$ and vice versa, we have that both $T_1$ and $T_2$ are TCTFs. Without loss of generality we assume that $\vass{P_1}\leq \vass{P_2}$. We know that $\vass{P_1}=2\vass{A_1}-4(1+2\epsilon)t>4\epsilon$. Let $xy$ be an edge in $P_2$, because every edge between $A_1$ and $A_2$ is blue, and because of the minimum degree condition we have that $x$ and $y$ share at least $\vass{P_1}-2\epsilon t$ blue neighbours in $P_1$. Because $P_1$ has a blue matching, every set in $P_1$ of size strictly bigger than $\frac{\vass{P_1}}{2}$ has an edge from $P_1$. Therefore we have that there exists $zt$ in $P_1$ such that $G[\llb x, y, z, t \rrb]$ is a blue clique with $xy$ in $P_2$ and $zt$ in $P_1$. Because both $P_1$ and $P_2$ are triangle connected, we are done. 

\smallskip

\underline{Case 4}: $G$ is not in any of cases 1--3, i.e.\ there is no red component of size $3(1+2\epsilon)t$ or bigger, and no two red components cover $(5+12\epsilon)t$ or more vertices.

Let $A_1, A_2, \dots{}$ be the maximal red connected components, ordered by decreasing cardinality. We have $\vass{A_1}<3(1+2\epsilon)t$ and  $\vass{A_1}+\vass{A_2}<(5+12\epsilon)t$, and we can assume that $G$ does not have a red connected matching over $2(1+\epsilon)t$ vertices since otherwise we are done.

\begin{claim}
The set of blue edges of $G$ is triangle connected.
\end{claim}
\begin{claimproof}
%
%
Every blue edge in a component $A_i$ is in a blue triangle with some vertex in a different component $A_j$, so it suffices to prove that the edges between distinct components all lie in the same triangle-connected component. In particular, it is enough to show that for any $j,k\ge 2$ distinct, any  $a_1a_j$ an edge between $A_1$ and $A_j$, and any $b_jb_k$ an edge between $A_j$ and $A_k$, then $a_1a_j$ and $b_jb_k$ are triangle connected. This last equivalence is due to the fact that there are at least three red components (indeed, $\vass{V(G)}-\vass{A_1\cup A_2}>(k-12)\epsilon t$).

Given $a_1,a_j,b_j,b_k$ as above, let $c$ be a common blue neighbour of $a_1, a_j, b_j$ not in $A_1\cup A_j$. This exists by minimum degree condition and by considering that $a_1, a_j, b_j$ are all in $A_1\cup A_j$ and there are at least $(k-12)\epsilon t$ vertices not in $A_1\cup A_j$. Now let us take $d$ a common blue neighbour of $c, a_j, b_j, b_k$ in $A_1$: this exists since $c,a_j,b_j,b_k$ are not in $A_1$, and using the minimum degree condition. We can now conclude since $(a_1a_jc, a_jcd, cdb_j, db_jb_k)$ is a sequence of blue triangles that proves that $a_1a_j$ and $b_jb_k$ are triangle connected.
\end{claimproof}
Because we showed that every blue edge is triangle connected, it is sufficient to find $(1+\epsilon)t$ disjoint blue triangles. We are in one of the following cases.

\underline{Case A}: Both $A_1$ and $A_2$ are larger than $2(1+20\epsilon)t$.

By Lemma \ref{lemmasteplowerbound1} we can find blue matchings $M_i\subseteq A_i$ on $2\vass{A_i}-4(1+2\epsilon)t$ vertices for $i=1,2$. Indeed, because $A_i< 3(1+2\epsilon)t$ we have $2\vass{A_i}-4(1+2\epsilon)t\leq \vass{A_i}-(1+2\epsilon)t$. We can greedily extend the matching $M_1$ to a blue triangle factor using vertices in $A_2\setminus M_2$: because $\vass{A_2\setminus M_2}=(4+8\epsilon)t-\vass{A_2}>\vass{A_1}-(2+4\epsilon)t+2\epsilon t= \frac{\vass{M_1}}{2}+2\epsilon t$ we are able to extend every edge in $M_1$ to a blue triangle. Similarly we can extend all the matching $M_2$ to a blue triangle factor using vertices in $A_1\setminus M_1$. This two triangle factors are disjoint and therefore they form a unique triangle factor that we denote with $T$. We can observe that $\vass{T}=\frac{3}{2}\vass{M_1}+\vass{M_2}=3(\vass{A_1}+\vass{A_2})-12(1+2\epsilon)t$.

Let us now denote $U_1=A_1\setminus T$, $U_2=A_2\setminus T$ and $W=V(G)\setminus (A_1\cup A_2)$. We have
\begin{align*}
  \vass{U_1}&=\vass{A_1}-\vass{M_1}-\frac{\vass{M_2}}{2}\\
  &=\vass{A_1}-2\vass{A_1}+4(1+2\epsilon)t+2(1+2\epsilon)t-\vass{A_2}\\
  &=6(1+2\epsilon)t-(\vass{A_1}+\vass{A_2})\geq t\,.  
\end{align*}

Similarly we have $\vass{U_2}\geq t$. We can also notice that $\vass{W}= (5+k\epsilon)t-(\vass{A_1}+\vass{A_2})\geq (k-12)\epsilon t$. Finally, let us observe that $\vass{U_1},\vass{U_2}> \vass{W}+4\epsilon t$ by our assumption on $|G|$. Therefore we can find a blue triangle factor on $(U_1,U_2,W)$ covering $3\vass{W}$ vertices. Adding this triangle factor to $T$ we get a TCTF on 
\[3(5+k\epsilon)t-3(\vass{A_1}+\vass{A_2})+3(\vass{A_1}+\vass{A_2})-12(1+2\epsilon)t=(3+(3k-24)\epsilon)t\]
 vertices.

\smallskip
\underline{Case B}: $A_1$ is larger than $2(1+3\epsilon)t$ but all the other red components are smaller than $2(1+3\epsilon)t$.

Let $M_1$ be a blue matching in $A_1$ on $2\vass{A_1}-4(1+2\epsilon)t$ vertices. Let $U_1=A_1\setminus M_1$ and notice $\vass{U_1}\geq 4(1+2\epsilon)t-\vass{A_1}$. Because all the other red components are smaller than $2(1+3\epsilon)t$, we claim there exists $j$ such that $(1+3\epsilon)t< \vass{\bigcup_{i=2}^{j}A_i}\leq 2(1+3\epsilon )t$, and write $U_2=\bigcup_{i=2}^{j}A_i$. Indeed, if $|A_2|> (1+3\epsilon )t$ we can take $j=2$, while if not then we can increase $j$ sequentially until the lower bound is satisfied. Since in the latter situation we have $|A_j|\le|A_2|\le (1+3\epsilon)t$ the upper bound is not exceeded. Finally, let $W=V(G)\setminus (A_1\cup U_2)$ and note that $\vass{W}\geq (3+(k-6)\epsilon)t-\vass{A_1}$.
    
    Because of the size of $U_2$, we can extend edges of the blue matching $M_1$ to form a triangle factor $T$ in $M_1\cup U_2$ over $3\vass{A_1}-6(1+2\epsilon)t$ vertices. We have that $\vass{U_2\setminus T}\geq 3(1+2\epsilon)t-\vass{A_1}$. Because $\vass{W},\vass{U_1}>\vass{U_2\setminus T}+4\epsilon$, we can find a blue triangle factor on $(U_1,U_2\setminus T, W)$ covering at least $3\vass{U_2\setminus T}$ vertices. Therefore combining this triangle factor with the one previously obtained over $M_1\cup T$ we have a TCTF over at least $3\vass{A_1}-6(1+2\epsilon)t +3(3(1+2\epsilon)t - \vass{A_1})= 3(1+2\epsilon)t$ vertices. 

\smallskip
\underline{Case C}: Assume all connected components are smaller than $2(1+2\epsilon)t$. This means that we can partition $V(G)$ in three sets $U_1,U_2$ and $W$ such that $(1+3\epsilon)t< \vass{U_1},\vass{U_2}\leq 2(1+3\epsilon)t$ by choosing unions of components as in the previous case to get $U_1$ and $U_2$, and let $W$ be the union of the remaining components. Thus there are no red edges between any two of $U_1, U_2$ and $W$.
Because $\vass{W}=5(1+k\epsilon)t-\vass{U_1}+\vass{U_2}$ we have that all three sets $U_1, U_2$ and $W$ have size at least $(1+3\epsilon)t$ and that the largest of the three has at least $(1+6\epsilon)t$ vertices. We can find a blue matching between the smallest two of $U_1,U_2,W$ greedily of size $(1+2\epsilon)t$, and extend this to a blue TCTF of size $3(1+2\epsilon)t$ vertices greedily, using the largest component.
\end{proof}

\begin{lemma}\label{tripartitetctf}
For $n\in\bb{N}$ sufficiently large, let $G$ be a tripartite graph over $3n$ vertices with partition sets of the same size. Assume that every vertex has at least $\frac{3n}{4}$ neighbours in each of the two partition sets of which it is not part of. There exists a TCTF that covers every vertex of $G$.

Also, every pair of edges in $G$ is triangle connected.
\end{lemma}
\begin{proof}
Let $m=\frac{3n}{4}$ and $X, Y$ and $Z$ denote the sets which partition $G$.
We first use Hall's theorem to prove that there exists a perfect matching $M$ between $X$ and $Y$. Indeed, let $S$ be a subset of $X$. If $\vass{S}\leq m$, because every vertex in $S$ has at least $m$ neighbours in $Y$ we have that the neighbourhood of $S$ in $Y$ has size not smaller than the size of $S$ itself. If $\vass{S}>m$ observe that by inclusion-exclusion principle we have that every vertex in $Y$ has a neighbour in $S$.
We shall now define a bipartite support graph $H$ over the sets $M, Z$. We add an edge between $xy$ and $z$ if the vertices $xyz$ form a triangle in $G$. We can observe that the existence of a perfect matching in $H$ gives us a triangle factor that covers all vertices of $G$.
Let $xy$ be in $M$, we can notice that since both $x$ and $y$ have at least $m$ neighbours in $Z$ we have that at least $\frac{n}{2}$ of the vertices of $Z$ are neighbours of both $x$ and $y$. Therefore every edge of $M$ has minimum degree at least $\frac{n}{2}$ in $H$.
Also, every vertex in $Z$ has minimum degree at least $\frac{n}{2}$ in $H$, since in $G$ it has minimum degree at least $m$ in both $X$ and $Y$.
We can then repeat the above piece of proof and use Hall's theorem to prove that we can find a perfect matching in $H$ and therefore a perfect triangle factor in $G$.

Let us now show that every couple of edges in $G$ is triangle connected. Let us first observe that if $xy$ and $xy'$ are both edges with $x\in X$ and $y,y'\in Y$ then we have that $x,y,y'$ share a neighbour in $Z$ and therefore they are triangle connected. This implies that the set of edges between $X$ and $Y$ is in the same triangle-connected component. We can easily conclude noticing that every triangle has one edge in each of the components $(X,Y), (Y,Z)$ and $(Z,X)$ which are therefore all the same triangle-connected component.  
\end{proof}

\begin{corollary}\label{corotripartitetctf}
For $n\in\bb{N}$ sufficiently large let $k,r\in\bb{N}$ such that $6r+4k<n$, let also $G$ be a tripartite graph over $3n$ vertices with partition sets $X, Y$ and $Z$ of the same size $n$. Moreover, assume every vertex in $G$ is adjacent to all but at most $k$ of the vertices in each of the two partition sets it is not a part of. Let us fix a 2-edge-colouring of $G$ such that $(X,Y)$, $(Y,Z)$ and $(X,Z)$ are $r$-red. We can find a red TCTF formed by at least $n-2r$ red triangles.

Also, all but at most $3r^2$ red edges of $G$ are in the same red triangle-connected component.
\end{corollary}
\begin{proof}
Let $X'\subseteq X, Y'\subseteq Y$ and $Z'\subseteq Z$ of size exactly $n'=n-2r$ such that every vertex in $X'\cup Y'\cup Z'$ has at most $r$ blue vertices in each of the other two components. We can apply Lemma \ref{tripartitetctf} to $G'=G^{Red}[X'\cup Y'\cup Z']$ considering that each vertex in $G'$ is adjacent to all but at most $r+k<\frac{3}{4}n'$ vertices in each of the two partitioning sets.  
\end{proof}

\begin{lemma}\label{lemmasuppforlemma4}
There exists $\epsilon_0\in\mathbb{R}$ such that for all $  0<\epsilon<\epsilon_0$ there exists $t_0$ such that for every $t>t_0$ we have the following. Let $G$ be a graph of minimum degree at least $\vass{G}-\epsilon t$ whose edges are 2-edge-coloured. If there exist in $G$ two disjoint sets $X$ and $Y$ of size respectively $(1+5\epsilon)t$ and $(5+200\epsilon)t$ such that $(X,Y)$ is $\epsilon t$-red, then $G$ contains a monochromatic TCTF on at least $3(1+\epsilon)t$ vertices.
\end{lemma}
\begin{proof}
Let $Y'$ be the set of vertices in $Y$ that have at least $|X|-\epsilon t$ red neighbours in $X$. We have that $\vass{Y'}\geq (5+100\epsilon)t$ and also $G[Y']$ has minimum degree at least $|Y'|-\epsilon t$. By Lemma \ref{lemmaconnmatchtctf} applied to $Y'$, we find either a blue TCTF of size $3(1+\epsilon)t$ or a red connected matching on $2(1+\epsilon)t$ vertices. In the first case we are done, so we can assume we have a red connected matching on $2(1+\epsilon)t$ vertices, let us denote it by $M$. By Lemma \ref{tripartitetctf} we can extend $M$ to a triangle factor $T$ of size at least $3(1+\epsilon)t$. We claim that this triangle factor is triangle connected. Indeed, every adjacent couple of red edges in $Y'$ is triangle connected since any three vertices in $Y'$ share a red neighbour in $X$. Since being triangle connected is a transitive property and because $M$ is red connected, we can conclude that $T$ is triangle connected.
\end{proof}

\section{General setting}\label{sec:genset}

To prove Lemma~\ref{MainLemma}, we will use a decomposition of $V(G)$ into red and blue cliques, and some associated notation. In this section, we describe the decomposition, define the notation, and prove that the decomposition exists under the assumptions of Lemma~\ref{MainLemma}.

\begin{setting}\label{mainsetforG}
Given $\epsilon,t>0$, let $m=\frac{1}{4}\vass{\log{\epsilon}}$ and let $k\in\mathbb{N}$ be arbitrary.

Given a graph $G$ with $(9-\epsilon)t$ vertices and minimum degree at least $(9-2\epsilon)t$, suppose that $E(G)$ is 2-edge-coloured and that there is no monochromatic TCTF with at least $3(1+\epsilon)t$ vertices.

We fix a partition of $V(G)$ into a set $\Vbin$ of size at most $\epsilon^{1/2}t+\frac{40t}{\sqrt{m}}$ and a collection of at most $\frac{9t}{m}$ monochromatic cliques each of size between $2$ and $m$ such that the following holds.

For each vertex $u$ which is in a blue clique $C$ of the partition, we assume that at most $\tfrac{20t}{\sqrt{m}}$ blue edges go from $u$ to vertices in blue cliques of the partition which are not blue triangle connected to $C$. We assume a similar statement replacing red with blue. Moreover, the number of cliques of size less than $(1-\frac{1}{k})m$ is at most $\frac{400k}{|\log\epsilon|^{3/2}}t$.

We write $B_1$ for a blue triangle-connected component of blue cliques of the partition covering the largest number of vertices, $B_2$ for the next largest, and so on. We break ties arbitrarily, and define similarly $R_1$ for the largest red triangle-connected component of red cliques of the partition and so on. We write $B_{\ge3}:=B_3\cup B_4\cup\dots$, and $R_{\ge3}:=R_3\cup R_4\cup\dots$.
\end{setting}

It is important to note that while we will care about which vertices contain the triangles of a TCTF, we will not care which vertices are used for the triangle connections between these triangles: when we ask whether two (say red) edges are red triangle-connected, we will always mean red triangle-connected in the entire graph $G$. Thus `there is a red TCTF in $X$ of size $3s$' means that there is a set of $s$ vertex-disjoint red triangles contained in the set $X$, which are all in the same red triangle component of $G$. In particular, the set $B_1$ is a collection of blue cliques which are blue triangle-connected in $G$, the connections might well use vertices outside $B_1$.

In the following sections, we will often state lemmas referring to a `decomposition as in Setting~\ref{mainsetforG}'. When we do this, we intend to fix a specific decomposition which will remain unchanged in the proof, and statements we make refer only to this decomposition. Thus `there is no red TCTF of size $3s$ contained in the red cliques' should be understood as meaning that the union of the red cliques \emph{of the fixed partition} do not contain such a TCTF. It might be that there is a different partition which does contain such a TCTF.

\medskip

The idea of our proof of Lemma~\ref{MainLemma} is now roughly as follows. We suppose that $G$ contains no large monochromatic TCTF. Our initial aim is then to show that each of $B_1,B_2,R_1,R_2$ has roughly $2t$ vertices, while $B_3\cup R_3$ has roughly $t$ vertices, these give us the five large sets of the partition of Lemma~\ref{MainLemma}. We will see that once the size bounds are obtained, it is not too hard to show that the edge colours are as claimed. Our proof for the claimed size bounds will go over several steps of finding increasingly strong upper and lower bounds on these sizes.

\medskip

We obtain Setting~\ref{mainsetforG} by iterative application of Ramsey's theorem followed by removing a few vertices to $\Vbin$. The following Lemma \ref{lem:mainsetexists} states that this is always possible, provided $\epsilon$ is small enough and $t$ large enough.

\begin{claim}\label{lemma3path}
For $n$ sufficiently large, let $G$ be a graph over $2n$ vertices, and let $A$, $B$ be disjoint cliques of size $n$ in $G$. If there are more than $2(n-1)$ edges between $A$ and $B$, the graph is triangle connected.
\end{claim}
\begin{proof}
Equivalently, we can show that if $H$ is subgraph of $K_{n,n}$ without a path of length three, then $H$ has at most $2(n-1)$ edges. Assume $H$ is a subgraph of $K_{n,n}$ without paths of length three. In particular this means that every edge has one endpoint with degree exactly one. Therefore the number of edges in $H$ is at most equal to the number of vertices in $H$ with degree one. If we have less than $2n-2$ vertices of degree one we are done. If we have $2n$ vertices with degree exactly one we know that $H$ is a perfect matching. It cannot be the case that $2n-1$ vertices have degree exactly one. Therefore we covered all cases and we can conclude that the number of edges in $H$ is at most $2(n-1)$.
\end{proof}

\begin{lemma}\label{lem:mainsetexists}
There exists $\epsilon_0\in\mathbb{R}$ such that for all $  0<\epsilon<\epsilon_0$ there exists $t_0$ such that for every $t>t_0$ and $k\in\mathbb{N}$ the following holds. Given a graph $G$ with at least $(9-\epsilon)t$ vertices, with minimum degree at least $(9-2\epsilon)t$, whose edges are 2-edge-coloured, there exist sets $R_1,\dots$ and $B_1,\dots$ of monochromatic red and blue cliques respectively satisfying the properties of Setting \ref{mainsetforG}.
\end{lemma}
\begin{proof}
Let us start by proving that we can find disjoint monochromatic copies of $K_m$ covering all but at most $\epsilon^{\frac{1}{2}}t$ vertices of $G$.

First, notice that we do not want all cliques to be of the same colour, we just want monochromatic cliques (some might be red, some might be blue).
Let us start by selecting greedily as many monochromatic copies of $K_m$ as possible in $G$, this means that we start by selecting an arbitrary monochromatic $K_m$, then we remove its vertices and we repeat the process over the remaining vertices of $G$.

Let us assume by contradiction that when this process stops more than $\epsilon^{\frac{1}{2}}t$ vertices of $G$ remain. Let $W$ be a set of size $\epsilon^{\frac{1}{2}}t$ not containing any monochromatic clique. Because of the minimum degree condition over $G$, we have that each vertex of $G[W]$ has degree at least $(\epsilon^{\frac{1}{2}}-\epsilon)t$ and therefore $G[W]$ contains at least $\epsilon^{\frac{1}{2}}(\epsilon^{\frac{1}{2}}-\epsilon)t^2=\left(1-\frac{1}{\epsilon^{-\frac{1}{2}}}\right)(\epsilon^{\frac{1}{2}}t)^2$ edges. 
By Turan's theorem, we have that $G[W]$ contains a (not necessarily monochromatic) clique $K$ of size $\epsilon^{-\frac{1}{2}}$. By a Ramsey's upper bound on diagonal Ramsey numbers we have that $R(m,m)\leq 4^m$, this value is smaller than $\epsilon^{-\frac{1}{2}}$ for $\epsilon$ small enough. Indeed, for $\epsilon<1$ we have $\epsilon=e^{-4m}$ and hence we can rewrite the inequality as $R(m,m)\leq 4^m\leq e^{2m}=\epsilon^{-\frac{1}{2}}$ which holds for $m$ large enough. Therefore we can find a monochromatic clique $K'$ in $W$.  
This contradicts the stopping of our greedy algorithm.

We can now focus on the number of vertices in blue cliques that witness more than $\frac{20t}{\sqrt{m}}$ blue edges that have endpoints in distinct triangle-connected components of blue $K_m$.
\begin{itemize}
    \item There are at most $\frac{9t}{m}$ disjoint copies of $K_m$ in $G$. This, combined with Claim \ref{lemma3path} gives us that at most $\frac{(10t)^2}{m}$ blue edges have endpoints in distinct triangle-connected components of blue $K_m$.
    \item At most $\frac{(20t)^2}{m}$ vertices in blue cliques of $G$ witness a blue edge with its two extremities in two distinct triangle-connected components of blue $K_m$. Therefore at most $\frac{20t}{\sqrt{m}}$ vertices in blue cliques witness more than $\frac{20t}{\sqrt{m}}$ such edges.
    \item We can do the same for red and obtain again at most $\frac{20t}{\sqrt{m}}$ vertices in red cliques that witness more than $\frac{20t}{\sqrt{m}}$ edges with their two extremities in two distinct triangle-connected components of red cliques.
\end{itemize}

 Finally, for some given positive integer $k$, we want to count how many monochromatic cliques in $V(G)\setminus \Vbin$ can have less than $(1-\frac{1}{k})m$ vertices. Which is, we want to state at most how many cliques of $G$ can have more than $\frac{m}{k}$ vertices in $\Vbin$. It is not difficult to see that this number is less than $\frac{40t}{\sqrt{m}}\cdot \frac{k}{m}\leq\frac{400 k }{\vass{\log{\epsilon}}^{\frac{3}{2}}}t$. 
Therefore at most $\frac{100 k }{\vass{\log{\epsilon}}^{\frac{1}{2}}}t$ vertices are in cliques of size at most $(1-\frac{1}{k})m$.
\end{proof}


\section{First upper bounds on the component size}
In this section, we prove that $|B_i|,|R_i|$ cannot be much bigger than $\tfrac73t$ (Lemma~\ref{upperboundsingle}) and that we cannot have both $B_1$ and $B_2$ (or $R_1$ and $R_2$) much bigger than $2t$ (Lemma~\ref{upperbounddouble}).

\begin{lemma}\label{upperboundsingle}
There exists $h_0>0$ such that for every $0<h<h_0$ there exists $\epsilon_0>0$ such that for all $0<\epsilon<\epsilon_0$ there exists $t_0$ such that for every $t>t_0$ we have the following. Let $G$ be a 2-edge coloured graph with $(9-\epsilon)t$ vertices and minimum degree at least $(9-2\epsilon)t$. Fix any collection of red and blue cliques as in Setting \ref{mainsetforG} with parameters $\epsilon$ and $t$. If $G$ has a set of blue triangle-connected cliques covering more than $(\frac{7}{3}+h)t$ vertices, then $G$ contains a monochromatic TCTF with $(1+\epsilon)t$ triangles. The same holds replacing blue with red.
\end{lemma}
\begin{proof}
Let $A$ be a triangle-connected set of blue cliques that covers more than $(\frac{7}{3}+h)t$ vertices. If $|A|\ge 3(1+\frac{50}{\vass{\log{\epsilon}}})t$ then we greedily construct a blue TCTF within $A$ that leaves out at most two vertices from each clique and obtain a blue TCTF covering at least $3(1+\epsilon)t$ vertices as desired, so we may now assume $|A|< 3(1+\frac{50}{\vass{\log{\epsilon}}})t$.

Because of this bound on the size of $A$, and the condition of Setting~\ref{mainsetforG} there are at most $\frac{40000}{\vass{\log{\epsilon}}^{\frac{3}{2}}} t$ cliques with less than $\frac{99}{100}m$ vertices, we have that there are at most
\[\frac{3(1+10\epsilon)t}{\frac{99}{100}m}+\frac{40000}{\vass{\log{\epsilon}}^{\frac{3}{2}}} t\leq \frac{16 t}{\vass{\log{\epsilon}}}\]
blue cliques in $A$ and at least $\vass{V(G)}-3(1+\frac{50}{\vass{\log{\epsilon}}})t$ vertices in $V(G)\setminus A$. 
In succession for each blue clique in $A$, we greedily construct a blue triangle factor $T$ using one edge in the selected clique and one vertex outside of $A$. There are two possible cases.

\underline{Case A}: The greedy construction provides us with a set $T$ of $\frac{2}{3}(1+\epsilon)t$ triangles. 

We can extend $T$ to a triangle factor $T'$ by adding triangles from within the cliques in $A$. When we stop, at most two vertices for each cliques are being unused and hence we obtained a blue TCTF covering at least 

$$3\cdot \frac{2}{3}(1+\epsilon)t + \left(\big(\tfrac{7}{3}+h\big) - 2\cdot \frac{2}{3}(1+\epsilon) - 2\cdot \frac{16}{\vass{\log{\epsilon}}}\right)t$$
vertices. Note that this means that $T'$ covers at least $3(1+\epsilon) t$ vertices.

\underline{Case B}: The greedy construction stops before we get  $\frac{2}{3}(1+\epsilon)t$ triangles. 

Let $Y=V(G)\setminus (A\cup T)$. We have that 
\begin{align*}
\vass{Y}&\geq (9-\epsilon)t-3(1+\frac{50}{\vass{\log{\epsilon}}})t-\frac{2}{3}(1+\epsilon)t\\
&\geq (5+h)t\ge\Big(5+\frac{20000}{\sqrt{\vass{\log\epsilon}}}\Big)t\,.    
\end{align*}
Let us denote by $X$ the set of all the vertices in $A\setminus T$ which are in cliques that have at least three vertices in $A\setminus T$. At most $\frac{4}{3}(1+\epsilon)t+2\cdot \frac{16}{\vass{\log{\epsilon}}}t$ vertices are in $A$ but not in $X$. Therefore we have that
\[\vass{X}\geq \left(1+\frac{100}{\sqrt{\vass{\log\epsilon}}}\right)t\,.\]
%
Because we stopped the greedy procedure, we cannot extend $T$ using an edge in a clique of $X$ and a vertex in $Y$, therefore each vertex in $Y$ has at most one blue neighbour in each clique of $X$.
This means that there are at most $\frac{16 t}{\vass{\log{\epsilon}}}\cdot \vass{Y}< \frac{16 t}{\vass{\log{\epsilon}}}\cdot\left( 9-\frac{7}{3} \right)t\leq \frac{20^2t^2}{\vass{\log{\epsilon}}}$ blue edges between $X$ and $Y$. Hence we have that $(X,Y)$ is $\frac{20}{\sqrt{\vass{\log{\epsilon}}}}t$-red. We can now apply Lemma \ref{lemmasuppforlemma4} with input $\frac{20}{\sqrt{\vass{\log\epsilon}}}$. We conclude that $G$ contains a monochromatic TCTF on at least
\[3\Big(1+\frac{20}{\sqrt{\vass{\log\epsilon}}}\Big)t>3(1+\epsilon)t\]
vertices.
\end{proof}


\begin{lemma}\label{upperbounddouble}
There exists $h_0>0$ such that for every $0<h<h_0$ there exists $\epsilon_0>0$ such that for all $0<\epsilon<\epsilon_0$ there exists $t_0$ such that for every $t>t_0$ we have the following. Let $G$ be a 2-edge coloured graph with $(9-\epsilon)t$ vertices and minimum degree at least $(9-2\epsilon)t$. Fix any collection of red and blue cliques as in Setting \ref{mainsetforG} with parameters $\epsilon$ and $t$. If $G$ contains two disjoint sets of blue triangle-connected cliques, with each set of cliques covering more than $(2+h)t$ vertices, then $G$ contains a monochromatic TCTF with $(1+\epsilon)t$ triangles. The same holds replacing blue with red.
\end{lemma}
\begin{proof}
Let $A$ and $B$ be disjoint sets of triangle-connected blue cliques, each covering at least $(2+h)t$ vertices. We may suppose $h_0\le\tfrac1{30}$. Let $C$ denote the collection of all the remaining vertices in blue cliques, if any exist. By Lemma \ref{upperboundsingle}, either we have the desired monochromatic TCTF or both $A$ and $B$ are smaller than $\big(\tfrac73+h\big)t$. Therefore by Setting~\ref{mainsetforG} with $k=100$ they both contain at most the following number of blue cliques:
\[\frac{\frac{71}{30}t}{\frac{99}{100}\cdot\frac{1}{4}\vass{\log{\epsilon}}}+\frac{40000t}{\vass{\log{\epsilon}}^{\frac{3}{2}}}\leq \frac{10t}{\vass{\log{\epsilon}}}\,.\]
Moreover, by Claim \ref{lemma3path} there are less than $2m$ blue edges between any blue clique in $A$ and any clique in $B$. Therefore, between $A$ and $B$ there are less than $2m\cdot \frac{10t}{\vass{\log{\epsilon}}} \cdot \frac{10t}{\vass{\log{\epsilon}}}\leq \frac{50}{\vass{\log{\epsilon}}}t^2$ blue edges. Hence, $(A,B)$ is $\frac{8}{\sqrt{\vass{\log{\epsilon}}}}t$-red. Let us set $\lambda=\frac{8}{\sqrt{\vass{\log{\epsilon}}}}$.

Let us greedily build a blue triangle factor $T_A$ by extending blue edges in blue cliques of $A$ to blue triangles using vertices outside of $A$. Let $Y_A$ be the set of vertices in $V(G)\setminus A$ used in this way
and $A'$ the set of remaining vertices in $A$
. We can independently do the same construction with $B$ and obtain a triangle factor $T_B$ and some similar sets $Y_B$ and $B'$. Finally, let us denote $Z=V(G)\setminus(A\cup Y_A\cup B\cup Y_B)$. 

Because we can extend $T_A$ to a blue TCTF that covers all but at most two vertices for each clique of $A$ {(and similarly for $B$)}, we have that $\vass{A\cup Y_A}, \vass{B\cup Y_B}\leq (3+3\epsilon+\frac{8}{m})t$. This implies $$\vass{Z}\geq \vass{V}-(\vass{A\cup Y_A}+\vass{B\cup Y_B})\geq (9-h)t-2(3+h)t=3(1-h)t\,.$$ 

We also have that $\vass{Y_A}, \vass{Y_B}\leq (1+\epsilon)t$, which implies that $\vass{A'}, \vass{B'} \geq (h-2\epsilon)t$.

Each vertex of $Z$ has at most one blue neighbour per clique in each of $A'$ and $B'$, since we cannot further extend $T_A$ or $T_B$. Since $\vass{Z}\leq 5t$ and there are at most $\frac{10 t}{\vass{\log{\epsilon}}}$ cliques in each of $A'$ and $B'$ we have that both $(A',Z)$ and $(B',Z)$ have at most $\frac{50}{\vass{\log{\epsilon}}}t^2$ blue edges and hence they are both $\lambda t$-red.

\begin{claim}
We claim that all red edges in $Z$ are triangle connected. Moreover, if $|C\cap Z|\ge \tfrac13t$ then we can find a TCTF in $(A, B, C\cap Z)$ on $\vass{C\cap Z}-ht$ triangles that is triangle connected to the red triangle-connected component of $Z$.
\end{claim}
\begin{claimproof}
Let $xy$ and $uv$ be two red edges in $Z$, let $N_A$ be the set of vertices in $A'$ red adjacent to all vertices $x,y,u$ and $v$ and let $N_B$ be defined similarly. To prove that $xy$ and $uv$ are triangle connected it suffices to show that there exists a red edge between $N_A$ and $N_B$. Because of the lower bound on the size of $A'$ and $B'$, because of the minimum degree condition and because every vertex in $Z$ is adjacent in red to all but at most $\frac{10t}{\vass{\log{\epsilon}}}$ of its neighbours in $A'$ and $B'$, we have that $\vass{N_A},\vass{N_B}\geq (h-2\epsilon)t-4\cdot \epsilon t-4\cdot \frac{10 t}{\vass{\log{\epsilon}}}\geq \frac{3h}{4}t$. Since $(A,B)$ is $\lambda t$-red, there is a red edge between $N_A$ and $N_B$. Therefore all the red edges in $Z$ are in the same triangle-connected component.

Let us now create a red TCTF $\Delta$ in $(A, B, C\cap Z)$ as follows. We first find a largest TCTF $\Delta'$ in $(A', B', C\cap Z)$. By Corollary \ref{corotripartitetctf}, we have that $\Delta'$ has at least $\frac{h}{2}t$ vertices, since we have a lower bound on both $\vass{A'}$ and $\vass{B'}$.

We can now use Corollary \ref{corotripartitetctf} to find a red TCTF in $(A\setminus \Delta', B\setminus \Delta', (C\cap Z)\setminus \Delta')$ that covers almost all $(C\cap Z)\setminus \Delta'$. Let us call $\Delta$ the union of the two triangle factors. By Lemma \ref{tripartitetctf} that $\Delta$ is triangle connected. 

It now suffices to show that $\Delta'$ is triangle connected to the red triangle-connected component of $Z$. Let $xy$ be a red edge in $Z$, let $N_A$ be the set of vertices in $A'\cap \Delta'$ red adjacent to both $x$ and $y$, and let $N_B$ be defined similarly in $B'\cap\Delta'$. To prove that $xy$ and $\Delta'$ are triangle connected it suffices to show that there exists an edge of $\Delta'$ between $N_A$ and $N_B$. Because every vertex in $Z$ is adjacent in red to all but at most $\frac{10t}{\vass{\log{\epsilon}}}$ of its neighbours in $A'$ and $B'$, we have that $\vass{N_A},\vass{N_B}\geq \frac{99h}{100}t$. Since $\Delta'$ is a matching in $(A', B')$ of large size, some of its edges are between $N_A$ and $N_B$.
\end{claimproof}

$Z\setminus C$ can be extended to a set of triangle-connected red cliques of $G$, possibly adding vertices from $Y_A$ and $Y_B$. Therefore, we have $\vass{Z\setminus C}\leq \left(\frac{7}{3}+h\right)t$ and this in particular implies that $\vass{C\cap Z}\geq (\frac{2}{3}-4h)t$. We form a red TCTF as follows. We start by using our last claim to construct a TCTF $T_C$ over at least $\vass{C\cap Z}-h t\geq (\frac{2}{3}-5h)t$ triangles between $A, B$ and $C\cap Z$ that is also triangle connected to the red triangle-connected component of $Z$. We then extend this TCTF by taking triangles in cliques of $Z\setminus C$. This is enough to conclude.
\end{proof}


\section{Colours and connection, and the sharp upper bound}
In this section we begin by proving two lemmas which show that certain patterns of edges between triangle components imply triangle connections, which we need in both this section and the next. We then establish several inequalities about sizes of the components (Lemma~\ref{biguniquelemma}), most of which imply that various components cannot be too small. In particular, we establish the useful inequality $|B_2|\ge|B_{\ge3}|$, and similarly for red. Building on this, we finally prove the sharp upper bound we want: none of the components can contain much more than $2t$ vertices (Lemma~\ref{lemmalastupperbound}). These are the two statements we need to complete the proof of Lemma~\ref{MainLemma} in the next section.

\subsection{Colours and connection}
\begin{claim}\label{coroappendix1}
For any $h>0$ there exists $\epsilon>0$ such that if we use that $\epsilon$ for Setting \ref{mainsetforG} we have the following. Let $A, B$ be two disjoint sets of vertices in blue cliques such that there are no triangle-connected components with some vertices in $A$ and some vertices in $B$. Then the pair $(A,B)$ is $ht$-red. The same works for red.
\end{claim}
\begin{proof}
By Remark \ref{lem:mainsetexists}, in $G$ there are at most $\frac{9t}{\frac{99}{100}m}+\frac{\const{4}t}{\vass{\log{\epsilon}}^{\frac{3}{2}}}\leq \frac{40t}{\vass{\log{\epsilon}}}$ cliques. Therefore, by Claim \ref{lemma3path} we can have at most $2m\cdot\frac{20t}{\vass{\log{\epsilon}}}\cdot \frac{20t}{\vass{\log{\epsilon}}}\leq \frac{200 t^2}{\vass{\log{\epsilon}}}$ blue edges between $A$ and $B$. In particular this means that the pair $(A,B)$ is $\sqrt{\frac{200}{\vass{\log{\epsilon}}}}t$-red. For $\epsilon$ small enough we have the result we wanted.
\end{proof}

\begin{lemma}\label{lemma17}
There exists $h_0>0$ such that for every $0<h<h_0$ there exists $\epsilon_0>0$ such that for all $0<\epsilon<\epsilon_0$ there exists $t_0$ such that for every $t>t_0$ we have the following. Let $G$ be a 2-edge coloured graph with $(9-\epsilon)t$ vertices and minimum degree at least $(9-2\epsilon)t$. Fix  any collection of red and blue cliques as in Setting \ref{mainsetforG} with parameters $\epsilon$ and $t$. Let $Y_1, Y_2, Y_3$ be subsets of size at least $10ht$ of vertices in distinct red triangle-connected clique components, and let $X$ be a set of size at least $ht$ of vertices in blue cliques which all have more than $2ht$ blue neighbours in two of the $Y_i$s. Then at least one of the blue edges in a clique of $X$ is triangle connected to the large blue TCTF in $(Y_1, Y_2, Y_3)$. Everything still works if we invert red and blue. 
\end{lemma}
\begin{proof}
First, note that for $\epsilon$ small enough and by Claim \ref{coroappendix1} we have that each pair in $Y_1, Y_2, Y_3$ is $\frac{h^3}{2}t$-blue. Let $R_i$ be the set of vertices in $Y_i$ with more than $h^3t$-red edges in one of the other $Y_j$.
Without loss of generality let us assume that the set $S$ of vertices in $X$ with more than $2ht$ blue neighbours in both $Y_1$ and $Y_2$ has size at least $\frac{ht}{3}$. Then each vertex in $S$ has at least $(2h-h^3)t$ blue neighbours in both $Y_1\setminus R_1$ and $Y_2\setminus R_2$. Then we have a vertex $y_1$ in $Y_1\setminus R_1$ which is incident in blue to at least $(2h-h^3)t\cdot\frac{ht}{3}\cdot\frac{1}{9t}\geq\frac{1}{15}h^2t$ vertices in $S$. So for $t$ large enough $y_1$ is incident in blue to at least two vertices of $S$ that lie in the same clique, let us call two such vertices $x_1$ and $x_2$. Since $y_1$ has at least $\vass{Y_2}-(\epsilon+h+h^3)t\geq \vass{Y_2}-(2h+h^3)t$ blue neighbours in $Y_2\setminus R_2$, we have that $y_1$ and $x_1$ have a common blue neighbour $y_2$. This implies that $x_1x_2$ is blue-triangle connected to $y_1y_2$ and this by minimum degree condition means that $x_1x_2$ is triangle connected to the large blue TCTF over $(Y_1, Y_2, Y_3)$ given by Lemma \ref{tripartitetctf}.
\end{proof}

\begin{lemma}\label{lemma19}
There exists $h_0>0$ such that for every $0<h<h_0$ there exists $\epsilon_0>0$ such that for all $0<\epsilon<\epsilon_0$ there exists $t_0$ such that for every $t>t_0$ we have the following. Let $G$ be a 2-edge coloured graph with $(9-\epsilon)t$ vertices and minimum degree at least $(9-2\epsilon)t$. Fix  any collection of red and blue cliques as in Setting \ref{mainsetforG} with parameters $\epsilon$ and $t$, and let $Y_1, Y_2$ be subsets of size at least $10ht$ of vertices in distinct red triangle-connected components, and let $X_1, X_2$ be subsets of size at least $10ht$ of vertices in distinct blue triangle-connected clique components.
Finally, assume that $X_1$ is $ht$-red to each of $Y_1$ and $Y_2$. Then at most $2ht$ vertices in $X_2$ have more than $2ht$ red neighbours in both $Y_1$ and $Y_2$. Everything still works if we invert red and blue.
\end{lemma}
\begin{proof}
First, note that for $\epsilon$ small enough and by Claim \ref{coroappendix1} we have that $(X_1,X_2)$ is $ht$-red.
Let $S$ be the set of vertices in $X_2$ which have more than $2ht$ red neighbours in both $Y_1$ and $Y_2$. Assume by contradiction $\vass{S}\geq 2ht$. Note that there is a vertex $x_1$ in $X_1$ which has at most $ht$ blue neighbours in each of $X_2$, $Y_1$ and $Y_2$, so $x_1$ is red-adjacent to some vertex $x_2\in S$. Now $x_1$ and $x_2$ have at least $\frac{h}{4}t$ common red neighbours in each $Y_i$ and therefore they have at least two common red neighbours from the same clique in each of the $Y_i$. But this is absurd because it would mean that a clique in $Y_1$ is triangle connected to a clique in $Y_2$.
\end{proof}

\subsection{Some lower bounds}
\begin{claim}\label{nuovoclaimsuireali}
Let $k$ be a positive integer and let $b_1\geq  \dots{}\geq  b_k>0$ be positive reals such that $\sum_{i> 1} b_i > b_1$. Then we can partition $\llb 1, \dots{}, k\rrb$ into two sets $A$, $B$ such that if $\alpha:=\sum_{i\in A} b_i$ and $\beta:=\sum_{i\in B} b_i$ we have $2\alpha\geq \beta\geq \alpha$.
\end{claim}
\begin{proof}
We can construct such a partition greedily in two steps. 

If $b_1+b_3\leq 2(b_2+b_4)$ we set $1, 3\in B$ and $2, 4\in A$. Otherwise we set $b_1\in B$ and $2, 3,\dots,\ell\in A$ with an $\ell$ such that $b_1>\sum_{i=2}^\ell b_i> \frac{b_1}{2}$ (such an $\ell$ exists because of the hypotheses and because $b_1>b_2+b_3$).

We now proceed by induction. Assume we already partitioned $1, \dots{}, i-1$ such that the requests of the lemma are satisfied and let $\alpha$ and $\beta$ be as in the statement of the lemma. If $2\alpha\geq \beta+b_i$ we can add $i\in B$. Otherwise, we have $\beta> 2\alpha-b_i \geq \alpha + b_i$, where the last inequality is given by the fact that the $b_i$ are ordered in decreasing order and $\vass{A}\geq 2$. In this second case we can add $i$ to the set $A$.
\end{proof}

\begin{lemma}\label{biguniquelemma}
There exists $h_0>0$ such that for every $0<h<h_0$ there exists $\epsilon_0>0$ such that for all $0<\epsilon<\epsilon_0$ there exists $t_0$ such that for every $t>t_0$ we have the following. Let $G$ be a 2-edge coloured graph with $(9-\epsilon)t$ vertices and minimum degree at least $(9-2\epsilon)t$. Fix  any collection of red and blue cliques as in Setting \ref{mainsetforG} with parameters $\epsilon$ and $t$, and define $B_1,B_2,\dots$ and $R_1,R_2,\dots$ as in Setting~\ref{mainsetforG}.
\begin{enumerate}[label=(\roman*)]
 \item\label{sublemma1} If $\vass{B_1}\leq \frac{7}{6}t$ then $\vass{\bigcup_{i}B_i}\leq \left(\frac{7}{2}+h\right)t$. 
 \item\label{sublemma2} If $\vass{B_1}\geq \frac{7}{6}t$ and $\vass{B_2}\leq \frac{7}{6}t$ then $\big|\bigcup_{i\neq 1}B_i\big|\leq \left(\frac{7}{3}+h\right)t$. 
 \item\label{sublemma3} If $\vass{B_1}, \vass{B_2}\geq \frac{7}{6}t$ then $\vass{\cup_{i}B_i}\leq (\frac{16}{3}+h)t$. 
 \item\label{sublemma4} We have $\frac{43}{12}t\leq \vass{\cup_i B_i}, \vass{\cup_i R_i}\leq (\frac{16}{3}+h)t$. We also have $\vass{B_1}> \frac{7}{6}t$.
 \item\label{sublemma5} If $\vass{B_2}<\vass{\cup_{i\geq 3} B_i}$ then we can find a red TCTF in $\cup_i B_i$ of size at least $\frac{3}{2}\vass{\cup_{i\geq 2}B_i}-ht$.
 \item\label{sublemma6} If $\vass{B_2}\leq \frac{8}{7}t$ then $\vass{\cup_{i}B_i}<(\frac{9}{2}-h)t$. Hence at most one of $B_2$ or $R_2$ can be smaller than $\frac{8}{7}t$.
 \item\label{sublemma7} We have $\vass{B_2}\geq \vass{\cup_{i\geq3} B_i}$ and $\vass{R_2}\geq \vass{\cup_{i\geq 3}R_i}$. 
\end{enumerate}
The corresponding results also hold for red and $R_1, R_2, R_3, \dots{}$.
\end{lemma}

\begin{proof}
We are going to prove these results in order, and we are sometimes going to use previous points already proved.

\underline{Proof or ~\ref{biguniquelemma}\ref{sublemma1}}: 
Suppose for a contradiction that $\vass{B_1}\leq\tfrac76t$ and $\vass{\bigcup_{i}B_i}> \left(\frac{7}{2}+h\right)t$. Observe that by Setting~\ref{mainsetforG} with $k=100$, all but at most $\frac{40000}{\vass{\log{\epsilon}}^{\frac{1}{2}}}t$ vertices of $G$ are in cliques fixed in Setting~\ref{mainsetforG} with at least $\frac{99}{100}m$ vertices. We let for each $i$ the set $B'_i$ consist of all vertices in blue cliques of $B_i$ with at least $\frac{99}{100}m$ vertices.

We want to study how many edges have endpoints in two distinct $B'_i$. For any fixed $i$, the maximum number of blue edges that have one endpoint in $B'_i$ and the other in some $B'_j$ with $j\neq i$, is less than
\[2m\cdot\frac{\vass{B'_i}}{\frac{99}{100}m}\cdot\frac{\vass{\bigcup_{j\neq i}B'_j}}{\frac{99}{100}m}\leq 3\frac{\vass{B'_i}\cdot \vass{\bigcup_{j\neq i}B'_j}}{m}\leq \frac{27}{m}t\vass{B'_i}\,.\]
Let us now observe that the number of vertices in $B'_i$ that have more than $\frac{h}{100}t$ blue neighbours outside of $B'_i$ is at most $\frac{27}{m}t\vass{B'_i}\cdot\frac{100}{ht}\leq \frac{10^4}{mh}\vass{B'_i}$.

Let us remove from each $B'_i$ all the vertices with more than $\frac{h}{100}t$ blue neighbours in $\bigcup_{j\neq i}B'_j$, let us call the result $B_i''$. By the last observation, we have that
\begin{align*}
 \Big|\bigcup_i B_i''\Big|&\geq \big(1-\tfrac{10^4}{mh}\big)\Big|\bigcup_i B'_i\Big|\\
 &\geq \big(1-\tfrac{10^4}{mh}\big)\Big(\Big|\bigcup_i B_i\Big|-\frac{40000}{\vass{\log{\epsilon}}^{\frac{1}{2}}}t\Big)\\
 &\geq \big(1-\tfrac{10^4}{mh}\big)\cdot \left(\tfrac{7}{2}+\tfrac{3h}{4}\right)t\\
 &\geq \big(\tfrac{7}{2}+\tfrac{3h}{4}-\tfrac{4\cdot 10^4}{mh}-\tfrac{10^4}{m}\big)t\\
 &\geq \big(\tfrac{7}{2}+\tfrac{h}{2}\big)t\,.
\end{align*}
 In $G^{Red}\big[\bigcup_i B''_i\big]$ every vertex has red degree at least $\vass{\bigcup_i B'_i}-(\frac{7}{6}+\epsilon+\frac{h}{100})t$ which is more than $\frac{2}{3}\vass{\bigcup_i B''_i}$. So by Lemma \ref{lemmatwothirdsappe}, $G^{Red}[\cup_i B_i'']$ contains a red TCTF of size $\frac{7}{2}t$.

\vspace{0.2cm}
\underline{Proof or ~\ref{biguniquelemma}\ref{sublemma2}}: 
Let $B^\ast_1$ be a set of the fixed blue cliques in $B_1$ covering between $\tfrac76t-m$ and $\tfrac76t$ vertices. We may assume $|B_2|\le |B^\ast_1|$, by swapping these two sets of cliques if necessary.

Repeating what we did in Lemma~\ref{biguniquelemma}\ref{sublemma1} to the sets $B^\ast_1,B_2,B_3,\dots$, we obtain $\big|B^\ast_1\cup\bigcup_{i\ge 2}B_i\big|\le\big(\tfrac72+h\big)t$. Since $|B^\ast_1|\le\tfrac76t$, we have $\big|\bigcup_{i\ge 2}B_i\big|\le\big(\tfrac73+h\big)t$ as desired.

\vspace{0.2cm}
\underline{Proof or ~\ref{biguniquelemma}\ref{sublemma3}}:
By Corollary \ref{corotripartitetctf} we have that $\vass{\cup_{i\geq 3} B_i}\leq (1+\frac{h}{3})t$ because otherwise we can find a red TCTF over more than $3(1+\epsilon)t$ vertices. By Lemmas \ref{upperboundsingle} and \ref{upperbounddouble} we have that $\vass{B_1}\leq (\frac{7+h}{3})t$ and $\vass{B_2}\leq (\frac{6+h}{3})t$.
Summing these bounds completes the proof.

\vspace{0.2cm}
\underline{Proof or ~\ref{biguniquelemma}\ref{sublemma4}}:
By Lemmas ~\ref{biguniquelemma}\ref{sublemma1}, \ref{sublemma2}, ~\ref{sublemma3} we have that for any possible size of $B_1$ and $R_1$ we always have $\vass{\cup_i B_i}, \vass{\cup_i R_i}\leq (\frac{16}{3}+h)t$. Because $\vass{\cup_i B_i}+\vass{\cup_i R_i}\geq (9-h)t$ we therefore must have $\frac{43}{12}t\leq \vass{\cup_i B_i}, \vass{\cup_i R_i}$. By Lemma ~\ref{biguniquelemma}\ref{sublemma1} this implies that $\vass{B_1},\vass{R_1}> \frac{7}{6}t$.

\vspace{0.2cm}
\underline{Proof or ~\ref{biguniquelemma}\ref{sublemma5}}:
Let us take a set of vertices $B_1'\subseteq B_1$ such that $\vass{B_1'}=\frac{1}{2}\vass{\cup_{i\geq 2} B_i}-\frac{1}{100}ht$ (we know that $B_1$ is large enough, indeed we know $\vass{B_1}\geq \frac{7}{6}t$ and it cannot be the case that $\vass{B_2}\geq \frac{7}{6}t$ because otherwise we would find a large red TCTF over $(B_1, B_2, \cup_{i\geq 3}B_i)$). By Claim \ref{coroappendix1} all but at most $\frac{1}{100}ht$ vertices of $B_1'$ have red degree in $G[B_1'\cup\bigcup_{i\geq 2}B_i]$ at least $\vass{\cup_{i\geq 2}B_i}-\frac{1}{100}ht$. Let $B_1''$ be a subset of size $\frac{1}{2}\vass{\cup_{i\geq 2} B_i}-\frac{2}{100}ht$ such that every vertex in $B_1''$ has red degree in $G[B_1''\cup\bigcup_{i\geq 2}B_i]$ at least $\vass{\cup_{i\geq 2}B_i}-\frac{1}{100}ht\geq \frac{2}{3}\vass{B_1''\cup\bigcup_{i\geq 2}B_i}$. Because every vertex in $\cup_{i\geq 2}B_i$ is in a triangle-connected component of size significantly smaller than $\frac{2}{3}\vass{B_1''\cup\bigcup_{i\geq 2}B_i}$ we can conclude by Lemma \ref{lemmatwothirdsappe} that we can find a red TCTF over all but at most two vertices of $B_1''\cup\bigcup_{i\geq 2}B_i$. Which is, we can find a red TCTF over at least $\frac{3}{2}\vass{\cup_{i\geq 2}B_i}-ht$ vertices.

\vspace{0.2cm}
\underline{Proof or ~\ref{biguniquelemma}\ref{sublemma6}}:
Fix some $h>0$ arbitrarily small, depending on which we can choose our $\epsilon$. By Lemma ~\ref{biguniquelemma}\ref{sublemma1} we can assume $\vass{B_1}\geq \frac{7}{6}t$. Remember also that we have by Lemma \ref{upperboundsingle}, $(\frac{7}{3}+h)t\geq \vass{B_1}$.
Assume by contradiction $\vass{B_2}\leq \frac{8}{7}t$ and $\vass{\cup_{i}B_i}\geq (\frac{9}{2}-h)t$. Then we would have $\vass{\cup_{i\geq 3}B_i}\geq (\frac{9}{2}-h)t-(\frac{7}{3}+h)t-\frac{8}{7}t=(\frac{43}{42}-2h)t$. By Corollary \ref{corotripartitetctf} and  Claim \ref{coroappendix1} it cannot be the case that $\vass{B_2}\geq (\frac{43}{42}-2h)t$ because otherwise we would find a large red TCTF over $(B_1, B_2, \cup_{i\geq 3}B_i)$. Therefore we must have $\vass{B_2}<\vass{B_3}$, and therefore by Lemma ~\ref{biguniquelemma}\ref{sublemma5} we must have that $\frac{3}{2}\vass{\cup_{i\geq 2}B_i}-ht< (3+h)t$ which is to say that $\vass{\cup_{i\geq 2}B_i}<\frac{25}{12}t$. We can conclude that $\vass{\cup_{i}B_i}<(\frac{7}{3}+h)t+\frac{25}{12}t<(\frac{9}{2}-h)t$.

\vspace{0.2cm}
\underline{Proof or ~\ref{biguniquelemma}\ref{sublemma7}}:
First, let us note that we cannot have both $\vass{B_2}< \vass{\cup_{i\geq3} B_i}$ and $\vass{R_2}< \vass{\cup_{i\geq3} R_i}$. Indeed, by ~\ref{biguniquelemma}\ref{sublemma6} at least one between $B_2$ and $R_2$ has cardinality at least $\tfrac{8}{7}t$. Let us say without loss of generality that $\vass{R_2}\geq \tfrac{8}{7}t$, then it cannot be $\vass{\cup_{i\geq 3}R_i}>\vass{R_2}$ because of Corollary \ref{corotripartitetctf}.

Let us now assume by contradiction that $\vass{\cup_{i\geq 3}B_i}>\vass{B_2}$. By Lemmas \ref{upperboundsingle} and \ref{upperbounddouble} we have that $\vass{R_1}\leq (\tfrac{7}{3}+h)t$ and $\vass{R_2}\leq (2+h)t$. Moreover, by Corollary \ref{corotripartitetctf} we have $\vass{R_3}\leq (1+h)t$. Therefore we have $\vass{B_2\bigcup\cup_{i\geq 3}B_i}\geq (\tfrac{4}{3}-5h)t$. 

By Claim \ref{nuovoclaimsuireali}, since both $B_3$ and $B_4$ are non-trivial (by our contradiction hypothesis), we can partition the sets $B_i$ into collections $B_1'$, $B_2'$ and $B_3'$ such that $B_1'=B_1$ and $\vass{B_2'}\geq \vass{B_3'}$ and also $\vass{B_2'}\leq \tfrac{2}{3}\vass{\cup_{i\geq 2}B_i}$. In particular this means $2\vass{B_3'}\geq \vass{B_2'}\geq \vass{B_3'}$ and $\vass{B_2'}\geq (\tfrac23-5h)t$ and $\vass{B_3'}\geq (\tfrac{4}{9}-5h)t$.

Notice that by Lemma~\ref{biguniquelemma}\ref{sublemma5} we have $\vass{B_2\bigcup\cup_{i\geq 3}B_i}\leq (2-2h)t$.
We claim that no blue clique in $B_1'$ is triangle connected to the blue TCTF in $(R_1, R_2, R_3)$. Indeed we have that this would create a blue TCTF of size at least $3\vass{R_3}+\vass{B_1}$ and we have $\vass{R_3}\geq 9t-\vass{B_1}-\vass{B_2\bigcup\cup_{i\geq 3}B_i}-\vass{R_1}-\vass{R_2}\geq(\tfrac13-5h)t$ and $\vass{R_3}+\vass{B_1}\geq(\tfrac{8}{3}-4h)t$. Which implies that $3\vass{R_3}+\vass{B_1}>(3+h)t$.

In particular, by Lemma \ref{lemma17} this implies that all but at most $ht$ vertices in $B'_1$ have less than $2ht$ blue neighbours in two of the $R_1, R_2$ or $R_{\geq 3}$. This means that there is a set $T\subset B_1'$ of size at least $\frac13\big(\vass{B_1'}-ht\big)$ such that $(T, R_i), (T, R_j)$ are $ht$-red and $i,j\in \llb 1,2,\geq 3\rrb$. Let us assume that $(T, R_2)$ is $ht$-red (if not, then we have $(T, R_1)$ is $ht$-red and this is strictly better in the following computations). We claim that $(R_2, B_2')$ and $(R_2, B_3')$ are $ht$-blue. Indeed by Lemma~\ref{biguniquelemma}\ref{sublemma5} and by the lower bound $\vass{B_2\bigcup\cup_{i\geq 3}B_i}\geq (\tfrac{4}{3}-5h)t$ we got earlier, we have a red TCTF in $B'_1\cup  B'_2\cup B'_3$ of size at least $(2-10h)t$, since $\vass{R_2}\geq \tfrac87t$ we must have that each clique in $R_2$ is not triangle connected to the large TCTF between the $B_i$ components. By Lemma \ref{lemma17} and since $(T, R_2)$ is red, we get that $(R_2, B_2')$ and $(R_2, B_3')$ are $ht$-blue.

We now claim that there is a $B_i$ in $B_2'$ such that $(B_i, R_{\geq3})$ is $ht$-red, in particular, this means that each red clique in $R_{\geq 3}$ is in the same triangle connected component of $(B'_1, B'_2, B'_3)$. There exists such a $B_i$ because $B_2'$ is formed by at least two distinct blue triangle-connected components, which cannot therefore be triangle connected among themselves. But we also know that $(B_2', R_2)$ is $ht$-blue, so if we had that more than one blue component in $B_2'$ has blue neighbours in $R_3$ we would get that these blue components are triangle connected.

Now we claim that we must have $\vass{R_{\geq 3}}\geq (1-20h)t$. As observed above, there is a red TCTF in $B'_1\cup B'_2\cup B'_3$ of size at least $\tfrac32\vass{B_2'\cup B_3'}$, and its triangles are triangle-connected in red to $R_{\ge3}$, we have a red TCTF of size $\tfrac32\vass{B_2'\cup B_3'}+\vass{R_3}-ht$, moreover, we have $\vass{R_3\cup B_3'\cup B_2'}\geq 9-\vass{B_1'\cup R_1\cup R_2}\geq (\tfrac{7}{3}-10h)t$ which gives us a red TCTF over more than $(3+h)t$ vertices, unless $\vass{R_3}\geq (1-20h)t$.

In particular, we can say that we can find a blue TCTF in $(R_1, R_2, R_{\geq 3})$ of size at least $3(1-20h)t$. Since we have already that $(R_2, B'_2)$ and $(R_2, B_3')$ are $ht$-blue, and since we cannot extend the blue TCTF in $(R_1, R_2, R_{\geq 3})$ at all, this means that $(R_1, B_2')$ and $(R_1, B_3')$ must be $ht$-red, but this is absurd since it would create a red TCTF in $(B_1', B_2', B_3')\cup R_1$ of size at least $\tfrac32\vass{B_3'\cup B_2'}+\vass{R_1}-ht>3(1+h)t$.
\end{proof}


\subsection{The sharp upper bound}
\begin{lemma}\label{lemmalastupperbound}
There exists $h_0>0$ such that for every $0<h<h_0$ there exists $\epsilon_0>0$ such that for all $0<\epsilon<\epsilon_0$ there exists $t_0$ such that for every $t>t_0$ we have the following. Let $G$ be a 2-edge coloured graph with $(9-\epsilon)t$ vertices and minimum degree at least $(9-2\epsilon)t$. Fix a  collection of red and blue cliques as in Setting \ref{mainsetforG} with parameters $\epsilon$ and $t$, and define $B_1,B_2,\dots$ as in Setting~\ref{mainsetforG}. We have that $\vass{B_1}, \vass{R_1}\leq (2+h)t$.
\end{lemma}
\begin{proof}
Let us denote with $B_{\geq 3}$ the set $\cup_{i\geq 3}B_i$ and similarly for red. By Lemmas \ref{upperboundsingle} and \ref{upperbounddouble} we can assume $\vass{B_1},\vass{R_1}\leq (\frac{7}{3}+h)t$ and $\vass{B_2},\vass{R_2}\leq (2+h)t$. Let us assume by contradiction that $\vass{B_1}\geq (2+h)t$. We construct greedily a blue TCTF $T_B$ as follows. Select an edge in a blue clique of $B_1$, and extend it (if possible) to a blue triangle using a vertex outside of $B_1$ not used yet in the process. We can repeat greedily until there are no blue edges in cliques of $B_1$ that can be extended outside of $T_B$. Let us denote with $Y_B$ the set of vertices $T_B\setminus B_1$ used to extend the edges in $B_1$, and let us denote with $B_1'$ the set $B_1\setminus T_B$ of remaining vertices.

Because $T_B$ is triangle connected, we have that the size of $T_B$ is smaller than $3(1+\epsilon)t$ and therefore in particular $\vass{B_1'}= \vass{B_1}-\vass{B_1\cap T_B}>\frac{h}{2}t$. 
Let $h':=\min\llb\frac{\vass{B'_1}}{200t}, h\rrb\geq h^{\tfrac32}$.

Because we stopped the greedy construction of $T_B$ only when we could not extend $T_B$ anymore, we have that every vertex in $V\setminus (B_1\cup Y_B)$ has at most as many blue neighbours in $B_1'$ as the number of cliques with at least two vertices that are in $B_1'$. This means that the number of blue edges in $(B_1',  V\setminus (B_1\cup Y_B))$ is at most $7t\cdot (\frac{7}{3}\frac{1}{\frac{99}{100}m}+\frac{k}{\vass{\log{\epsilon}}^{\frac{3}{2}}})t\leq (\frac{5}{\sqrt{m}}t)^2$. Therefore we have that the pair $(B_1',  V\setminus (B_1\cup Y_B))$ is $\lambda t$-red for $\lambda=\frac{5}{\sqrt{m}}$.

We now separate four cases.

\underline{Case A}: We already have $\vass{B_1}\geq (2+h)t$, and $\vass{B_1},\vass{R_1}\leq (\frac{7}{3}+h)t$, and $\vass{B_2},\vass{R_2}\leq (2+h)t$. Assume now $\vass{R_{\geq 3}}\leq ht$.

It follows that $\vass{\cup_{i} R_i}\leq (\frac{13}{3}+3h)t$, and from this it follows that $\vass{\cup B_i}\geq (\frac{14}{3}-4h)t$ and therefore $\vass{B_2\cup B_{\geq 3}}\geq (\frac{7}{3}-5h)t$. Since $\vass{B_2}\geq\vass{B_{\geq 3}}$ by ~\ref{biguniquelemma}\ref{sublemma7},  we have that $\vass{B_2}>(1+h)t$ and therefore by \ref{tripartitetctf}
 we have $\vass{B_{\geq 3}}<(1+h)t$. By what stated above, we also get:
 
we get the following:
    \begin{align*}\vass{B_{\geq 3}}\geq 
        \begin{cases}
            (\frac{8}{3}-5h)t-\vass{R_1}\\
            (\frac{7}{3}-5h)t-\vass{R_2}\geq (\frac{1}{3}-6h)t\\
            (\frac{14}{3}-4h)t-\vass{B_1}-\vass{R_2}
        \end{cases}
    \end{align*}
    Since $\vass{B_{\geq 3}}<(1+h)t$, we must have $\vass{R_2}>(\frac{4}{3}-6h)t$.
    Let us call $C_B$ the red triangle-connected component in $(B_1, B_2, B_{\geq 3})$ that by Corollary \ref{tripartitetctf} contains almost all red edges of $(B_1, B_{\geq 3})$ and $(B_2, B_{\geq 3})$. 
    \begin{claim}
    No red edge in a clique of $R_1$ or $R_2$ is red triangle connected to $C_B$.
    \end{claim}
    \begin{claimproof}
    If $R_1$ were red triangle connected to $C_B$ we could extend a large red TCTF of size $3\vass{B_{\geq 3}}-ht$ (which is given us by Corollary \ref{tripartitetctf}) using vertices of $R_1$ and obtain a TCTF over 
   \begin{align*}  
   3\vass{B_{\geq 3}}+\vass{R_1}-2ht &\geq 3((\frac{8}{3}-5h)t-\vass{R_1})+\vass{R_1}-2ht\\
   &=(8-17h)t-2\vass{R_1}>3(1+\epsilon)t
   \end{align*}
    vertices. It is also absurd that $R_2$ is red triangle connected to $C_B$. Indeed we would have:

\underline{\textit{Case 1}}: If $\vass{R_2}\leq \frac{17}{9}t$ then we have a red TCTF over $3\vass{B_{\geq 3}}+\vass{R_2}-ht\geq (7-16h)t-2\vass{R_2}>3(1+\epsilon)t$ vertices.

\underline{\textit{Case 2}}:  If $\vass{R_2}\geq \frac{17}{9}t$, we can greedily construct a TCTF $T$ as follows. We select edges in red cliques of $R_2\setminus Y_B$ and we extend them to disjoint triangles using vertices of $B_1'$. Because $(R_2\setminus Y_B,B_1')$ is $\lambda t$-red we have that we can continue this process until we almost finish red edges in red cliques of $R_2\setminus Y_B$ (we can have at most $ht$ vertices remaining in $R_2\setminus Y_B$) or vertices in $B_1'$ with enough red neighbours in $R_2\setminus Y_B$ (if we stopped because of this we have that at most $h't$ vertices in $B_1'$ are not used, because otherwise we would have vertices with high red degree in $R_2\setminus Y_B$ and because $\vass{R_2\setminus Y_B}$ would be at least $h't$ we could find an edge in a clique between two neighbours of the same vertex of $B_1'$). At this point we can extend $T$ with triangles from cliques of $R_2$ and obtain a TCTF over at least $\min\llb \vass{R_2}+\vass{B_1'}, \frac{3}{2}\vass{R_2\setminus Y_B}+\vass{R_2\cap Y_B}\rrb-3ht$ vertices. $T$ intersects $B_1$ in at most $t$ vertices, therefore we can again extend $T$ using the tripartition $(B_1\setminus T, B_2, B_{\geq 3})$, in this way we are adding at least $3\vass{B_{\geq 3}}-ht$ vertices since $\vass{B_3}\leq (1+h)t$, $\vass{B_2}> \frac{7}{6}$ and $\vass{B_1}\geq (2+h)t$. Therefore we end up with a red TCTF over $\min\llb \vass{R_2}+\vass{B_1'}, \frac{1}{2}\vass{R_2\setminus Y_B}+\vass{R_2}\rrb+3\vass{B_{\geq 3}}-4ht=3\vass{B_{\geq 3}}+\vass{R_2}+\min\llb \vass{B_1'}, \frac{1}{2}\vass{R_2\setminus Y_B}  \rrb-4ht$ vertices. We can notice at this point that 
\begin{align*}
\vass{B_1'}+\vass{B_{\geq 3}}&\geq \vass{B_{1}}+\vass{B_{\geq 3}}-(2+h)t\\
&\geq \frac{14}{3}t-4ht-\vass{B_2}-(2+h)t\\
&\geq (\frac{2}{3}-6h)t\,.    
\end{align*}
Since $\vass{B_{\geq 3}}\geq (\frac{1}{3}6h)t$ and $\vass{R_2\setminus Y_B}\geq \frac{5}{6}t$, we are done. Indeed we have $3\vass{B_{\geq 3}}+\vass{R_2}+\min\llb \vass{B_1'}, \frac{1}{2}\vass{R_2\setminus Y_B}  \rrb-4ht\geq 3(\frac{1}{3}6h)t+\frac{17}{9}t+\frac{1}{3}t>3(1+h)t$.
    \end{claimproof}

    Now we know that neither $R_1$ nor $R_2$ are triangle connected to the large triangle-connected component of the tripartition $(B_1, B_2, B_{\geq 3})$.
    In order to use Lemma \ref{lemma17} efficiently, we first need to remember that $B_1', R_1\setminus Y_B$ and $R_2\setminus Y_B$ are all non-trivial and that $(B_1',R_1\setminus Y_B)$ and $(B_1',R_2\setminus Y_B)$ are both $\lambda t$-red. Now we can use Lemma \ref{lemma17} to conclude that at most $2\lambda t$ vertices in $(R_1\cup R_2)\setminus Y_B$ can have more than $2\lambda t$ red neighbours in each of $B_2$ and $B_{\geq 3}$.
    But this is absurd because of Lemma \ref{lemma19}.
    
\underline{Case B}: We already have $\vass{B_1}\geq (2+h)t$, and $\vass{B_1},\vass{R_1}\leq (\frac{7}{3}+h)t$, and $\vass{B_2},\vass{R_2}\leq (2+h)t$. Assume now that $\vass{B_{\geq 3}}\leq ht$.

We can also assume that $\vass{R_1}\leq (2+h)t$ because otherwise we would be in the same situation as case A under switching colours. By Corollary ~\ref{biguniquelemma}\ref{sublemma4} we have $\vass{\cup_i B_i}\geq \frac{43}{12}t$ and this implies $\vass{B_2}\geq \frac{6}{5}t$. We can consider that $\vass{R_2\cup R_{\geq 3}}\geq (9-h)t-\vass{R_1}-\vass{\cup_i B_i}\geq (\frac{8}{3}-5h)t$, which gives us $\vass{R_{\geq 3}}\geq (\frac{2}{3}-6h)t$. By Lemma ~\ref{biguniquelemma}\ref{sublemma7} we have $\vass{R_2}\geq (\frac{4}{3}-3h)t$. By Corollary \ref{corotripartitetctf} this also implies that there is a red TCTF on $(R_1, R_2, R_{\geq 3})$ covering at least $3\vass{R_{\geq 3}}-ht\geq (2-19h)t$ vertices. This gives us the upper bound $\vass{R_{\ge3}}\le \frac{1+h}t$. This also implies that $\vass{R_2}\geq (\tfrac53-6h)t$.

Since both $B_1$ and $B_2$ are bigger than $\frac{8}{7}t$ we have that neither $B_1$ nor $B_2$ can be blue triangle connected to the large TCTF over $(R_1, R_2, R_{\geq 3})$. 
    
    By Lemma \ref{lemma17} this means that at most $ht$ vertices from each of $B_1$ and $B_2$ can be blue adjacent to more than $2ht$ vertices in any two of $R_1, R_2$ or $R_{\geq 3}$. But we know also that $B'_1$, $R_1\setminus Y_B$ and $R_2\setminus Y_B$ are non trivial, and therefore $(B_1', R_1\setminus Y_B)$ and $(B_1', R_2\setminus Y_B)$ are $\lambda t$-red. Hence, by Lemma \ref{lemma19} it can not not be the case that there are more than $2ht$ vertices of $B_2$ with more than $2ht$ red neighbours in both $R_1\setminus Y_B$ and $R_2\setminus Y_B$. Therefore by Lemma \ref{lemma17} there are at most $3ht$ vertices in $B_2$ which have more than $2ht$ blue neighbours in $R_{\geq 3}$.
    
    This means that we can find a set $S_1$ of at least $\frac{1}{2}\vass{B_2}-10ht$ vertices in $B_2$ such that every vertex in $S_1$ has at most $2ht$ blue neighbours both in $R_{\geq 3}$ 
    and one of $R_1\setminus Y_B$ or $R_2\setminus Y_B$ (say $R_2$, it is the same if it was $R_1$). Therefore by applying Lemma \ref{lemma19} with $S_1$ and $B_1'$ on one side and $R_2$ and $R_{\geq 3}$ on the other side, we get that there are at most $6h't$ vertices in $B_1'$ which have more than $3h't$ red neighbours in $R_{\geq 3}$, and this means that $(B_1', R_{\geq 3})$ is $6h't$-blue. 
    By Lemma \ref{lemma17} we know that $(B_1', R_1)$ and $(B_1',R_2)$ are $9h't$-red, and in the same way we know that almost all the vertices of $B_2$ are $2h't$-red to one of $R_1$ or $R_2$. As an example, we are going to assume that we have a subset $S_2$ of $B_2$ of size at least $\frac{\vass{B_2}-20h't}{2}$ such that every vertex in $S_2$ has at most  $2h't$ blue neighbours to $R_2$.
    
    Therefore $(S_2, R_{\geq 3})$ and $(S_2, R_2)$ are $2h't$-red. Because $(B_1', R_1)$ and $(B_1', R_2)$ are both $9h't$-red, by lemma \ref{lemma19} we have that $(B_1, R_1)$ is $9h't$-red.
    By Lemma \ref{lemma17} as above, at most $6h't$ vertices in $B_1$ can have more than $2h't$ blue neighbours in any two of $R_1$, $R_2$ and $R_{\geq3}$. We can find $S'\subseteq B_1$ of size at least $\frac{\vass{B_1}-20h't}{2}$ that is either $10h't$-red to $R_{\geq 3}$ or to $R_{2}$.
    In the first case, we find a large red TCTF using triangles in $(S', S_2, R_{\geq 3})$ and then triangles in $B_1$. In the latter case, we can find a red TCTF on $(S_2, S', R_2)$ over at least $$2\cdot \min\llb\vass{S_2}, \vass{S'}, \vass{R_2}\rrb+\vass{R_2}-20h't$$ 
    vertices. We claim that this is enough, indeed we have $\vass{R_{\geq 3}}\leq (1+h')t$, and therefore we get the lower bound $(\frac{5}{3}-10h')t$ for $\vass{R_2}$ and $t-10h't$ for $\vass{S'}$. 
%
%
%
%
    
\underline{Case C}: We already have $\vass{B_1}\geq (2+h)t$, and $\vass{B_1},\vass{R_1}\leq (\frac{7}{3}+h)t$, and $\vass{B_2},\vass{R_2}\leq (2+h)t$. Assume now $\vass{R_1}\leq (2+h)t$ and $\vass{B_{\geq 3}},\vass{R_{\geq 3}}\geq ht$.

We have two cases.

\underline{\textit{Case C.1}}: Let us assume $\vass{R_2}\leq \frac{8}{7}t$. 
\begin{claim}
Neither $B_1$ nor $B_2$ is blue connected to the TCTF over $(R_1, R_2, R_{\geq 3})$. Also, $R_1$ is not triangle connected to $(B_1, B_2, B_{\geq 3})$.
\end{claim}
\begin{claimproof}
By Corollary ~\ref{biguniquelemma}\ref{sublemma4} we have that $\vass{R_1\cup R_2\cup R_{\geq 3}}\geq \frac{43}{12}t$ and hence $\vass{R_2\cup R_{\geq 3}}\geq (\frac{19}{12}-h)t$. 

    
By Lemma~\ref{biguniquelemma}\ref{sublemma7} we have $\vass{R_2}\geq\vass{R_{\geq 3}}$ and by Lemma~\ref{biguniquelemma}\ref{sublemma6} we have $\vass{B_2}\geq \frac{8}{7}t$ and since $R_1,\dots,B_{\ge3}$ form a partition of $G$, we have $\vass{R_{\geq 3}}\geq (9-\frac{7}{3}-1-2-\frac{8}{7}-3h)t-\vass{B_2}>(\frac{5}{2}+h)t-\vass{B_2}$. By Corollary \ref{tripartitetctf} we can find a blue TCTF over $(R_1, R_2, R_{\geq 3})$ of size at least $3\vass{R_{\geq 3}}\geq\frac{15}{2}t-3\vass{B_2}$. In particular this implies that both $B_1$ and $B_2$ are not triangle connected to the blue TCTF over $(R_1, R_2, R_{\geq 3})$.
    
 We now prove that $R_1$ is not triangle connected to $(B_1, B_2, B_{\geq 3})$. 
 
 If $\vass{R_2\cup R_{\geq 3}}> (\frac{8}{7}+1+h)t$, 
 then by Lemma~\ref{biguniquelemma}\ref{sublemma7} we have  $\vass{R_2}>\vass{R_{\ge3}}$, since $\vass{R_2}\le\tfrac87t$ we have $\vass{R_{\ge3}}\ge(1+h)t$ and by Corollary~\ref{corotripartitetctf} we again obtain a blue TCTF of size $(3+h)t$.
 
If on the other hand we have $\vass{R_2\cup R_{\geq 3}}\leq (\frac{8}{7}+1+h)t$, it follows that $\vass{B_{\geq 3}}\geq (9-\frac{7}{3}-2-\frac{8}{7}-1-4h)t-\vass{R_1}$ which means $3\vass{B_{\geq 3}}+\vass{R_1}\geq \frac{24}{7}t$. Therefore it cannot be that $R_1$ is red triangle connected to the large TCTF over $(B_1, B_2, B_{\geq 3})$.
\end{claimproof}

Since $R_1$ is not connected to $(B_1, B_2, B_{\geq 3})$ we have by Lemma \ref{lemma17} that at most $h^5t$ vertices in $R_1$ have more than $2h^5t$ red neighbours in two of the $B_i$. Since $R_1\setminus Y_B$ is nontrivial we have that $(B_1', R_1\setminus Y_B)$ is $\lambda t$-red. Therefore we must have that $(R_1\setminus Y_B, B_2), (R_1\setminus Y_B,B_{\geq 3})$ are $h^2t$-blue. We can now apply Lemma \ref{lemma17} again knowing that $B_2$ is not blue triangle connected to the blue triangle component over $(R_1, R_2, R_{\geq 3})$ and therefore at most $h^5t$ vertices of $B_2$ have more than $2h^5t$ blue neighbours in two of the $R_i$. Hence, $(B_2, R_2)$ and $(B_2,R_{\geq 3})$ are $h^2t$-red.
    
Since they are not red triangle connected among themselves, we have that either $R_2$ or $R_{\geq 3}$ is not red triangle connected to the red triangle component over $(B_1, B_2, B_{\geq 3})$. Let $R_{2}$ the one not red triangle connected, and $R_{\geq 3}$ the other one (if the situation is reversed we get better bounds). Then by Lemma \ref{lemma17} we have that $R_2$ is $h^2t$-blue to $B_1$ and $B_{\geq 3}$, and therefore by the same Lemma we have that $(B_1, R_{\geq 3})$ is $h^2t$-red. Then $(B_1, B_2, B_{\geq 3}\cup R_{\geq 3})$ is a dense red tripartition with $\vass{B_1}, \vass{B_2}\geq\frac{8}{7}t$. We have $\vass{B_{\geq 3}\cup R_{\geq 3}}\geq (9-\frac{7}{3}-2-2-\frac{8}{7}-3h)t\geq \frac{3}{2}t$ which is enough to conclude by Corollary \ref{tripartitetctf}.
    
\underline{\textit{Case C.2}}: Let us now assume $\vass{R_2}\geq \frac{8}{7}t$. 

Then both $R_1\setminus Y_B$ and $R_2\setminus Y_B$ are non trivial and $\lambda t$-red to $B_1'$. We cannot have that both $R_1$ and $R_2$ are red triangle connected to $(B_1, B_2, B_{\geq 3})$ (because otherwise they would be red triangle connected among themselves). By Lemma \ref{lemma17} this means that one between $R_1\setminus Y_B$ and $R_2\setminus Y_B$ must be $h^2t$-blue to both $B_2$ and $B_{\geq 3}$, we are going to work with the example in which $R_1\setminus Y_B$ is $h^2t$-blue to both $B_2$ and $B_{\geq 3}$ (it would be the same if we had $R_2$).
    
    We cannot have both $B_2$ and $B_{\geq 3}$ to be blue triangle connected to $(R_1, R_2, R_{\geq 3})$ (otherwise they would be in the same connected component) and therefore we split our case depending on whether or not $B_2$ is blue triangle connected to $(R_1, R_2, R_{\geq 3})$. 
    
    Let us assume that it is so. Then $B_{\geq 3}$ is not blue triangle connected to $(R_1, R_2, R_{\geq 3})$ and so $(B_{\geq 3},R_2)$ and $(B_{\geq 3}, R_{\geq 3})$ are $h^2t$-red. By Lemma \ref{lemma17} this implies that $R_2$ is red triangle connected to $(B_1, B_2, B_{\geq 3})$ and therefore $R_{\geq 3}$ is not. Therefore $(R_{\geq 3}, B_1)$ and $(R_{\geq 3}, B_2)$ are $h^2t$-blue. Therefore $(B_1, R_1)$ and $(B_1,R_2)$ must be $h^2t$-red. Therefore we can find a blue TCTF over $3\vass{R_{\geq 3}}+\vass{B_2}$ vertices by taking triangles from $(R_1, R_2, R_{\geq 3})$ and $B_2$. We can also find a red TCTF over $3\vass{B_{\geq 3}}+\frac{3}{2}\vass{R_2}$ vertices by taking triangles from $(B_1, B_2, B_{\geq 3})$ and by taking edges from $R_2$ and extending them with vertices from $B_1$. We conclude by taking the average of the size of these two TCTFs.
    
    Let us now assume that $B_2$ is not blue triangle connected to $(R_1, R_2, R_{\geq 3})$. Then $(B_2, R_2)$ and $(B_2, R_{\geq 3})$ are $h^2t$-red, since $(R_1\setminus Y_B, B_2)$ is $ht$ -blue, and this implies that $R_2$ is red triangle connected to $(B_1, B_2, B_{\geq 3})$. Therefore $R_{\geq 3}$ is $h^2t$-blue to $B_1$ and $B_{\geq 3}$ and so $B_{\geq 3}$ is blue triangle connected to $(R_1, R_2, R_{\geq 3})$. This also means that $B_1$ must be $h^2t$-red to both $R_1$ and $R_2$ in order not to be blue triangle connected to $(R_1, R_2, R_{\geq 3})$ but this leaves us with a dense red $(B_1, B_2, B_{\geq 3}\cup R_2)$.

\underline{Case D}: Let us assume that $\vass{R_1}\geq (2+h)t$ and both $B_{\geq 3}$ and $R_{\geq 3}$ contain more than $ht$ vertices (otherwise we are in the situation of case B). We can also assume without loss of generality that $\vass{B_{2}}\geq \vass{R_2}$ and therefore by Lemma \ref{biguniquelemma}\ref{sublemma6} we also have $\vass{B_2}\geq \tfrac{8}{7}t$.

We can greedily extend blue edges in cliques of $B_1$ to a blue TCTF $T_B$ by using vertices outside of $B_1$. Since $\vass{B_1}\geq (2+h)t$ we can either create a TCTF over more than $3(1+\epsilon)t$ vertices or we have to stop at some point. Since $\vass{B_1}\geq (2+h)t$, this means that $B_1\setminus T_B$ is non trivial. We can do the same with a red TCTF $T_R$ extending red edges in $R_1$ (since we are assuming $\vass{R_1}\geq (2+h)t$). Let us call $B_1':=B_1\setminus T_B$ and $R_1':=R_1\setminus T_R$. Since the TCTFs $T_R$ and $T_B$ are maximal, we have that $(B_1', V(G)\setminus T_B)$ is $ht$-red, while $(R_1, V(G)\setminus T_R)$ has to be $ht$-blue. In particular, there are non-trivial subsets $S_{B_1}\subseteq B_1$ of size at least $(1+\tfrac{h}{2})t$, $S_{B_2}\subseteq B_2$ of size at least $(\tfrac{1}{7}+\tfrac{h}{2})t$ and $S_{R_1}\subseteq R_1$ of size at least $(1+\tfrac{h}{2})t$ such that $(S_{B_1}, R_1')$ and $(S_{B_2}, R_1')$ are $ht$-blue and $(S_{R_1}, B_1')$ is $ht$-red.

There are two cases:



\underline{Case D.1}: $B_1$ is blue triangle connected to the large TCTF in $(R_1, R_2, R_{\geq 3})$. Then we know that $B_2$ and $B_{\geq 3}$ are not triangle connected to the same TCTF. In particular, since $(S_{B_2}, R_1')$ is $ht$-blue, we must have that both $(S_{B_2}, R_2)$ and $(S_{B_2}, R_{\geq 3})$ are $ht$-red. Now, either $R_1$ is red triangle connected to the large TCTF in $(B_1, B_2, B_{\geq 3})$ or not. 

In the first case, we have that both $R_2$ and $R_{\geq 3}$ are not triangle connected to the large TCTF in $(B_1, B_2, B_{\geq 3})$. Because $(S_{B_2}, R_2)$ and $(S_{B_2}, R_{\geq 3})$ are $ht$-red, this means that $(R_{\geq 3}, B_{\geq 3}), (R_{\geq 3}, B_1)$ and $(R_{2}, B_{\geq 3}), (R_{2}, B_{1})$ are $ht$-blue, which is absurd because it would mean that $B_1$ and $B_{\geq 3}$ are in the same blue-connected component.

If $R_1$ is not red triangle connected to the large TCTF in $(B_1, B_2, B_{\geq 3})$, then $(S_{R_1}, B_{\geq 3})$ and $(S_{R_1}, B_2)$ have to be $ht$-blue. But now we get a contradiction since $(B_{\geq 3}, R_{\geq 3})$ and $(B_{\geq 3}, R_{2})$ need to be $ht$-red
 or otherwise $B_{\geq 3}$ is going to be triangle connected to the blue TCTF in $(R_1, R_2, R_{\geq 3})$, and also $(B_{2}, R_{\geq 3})$ and $(B_{2}, R_{2})$ need to be $ht$-red
 or otherwise $B_{2}$ is going to be triangle connected to the blue TCTF in $(R_1, R_2, R_{\geq 3})$.  This is enough to say that $R_{2}$ and $R_{\geq 3}$ are in the same red-connected component.

\underline{Case D.2}: $B_1$ is not blue triangle connected to the large TCTF in $(R_1, R_2, R_{\geq 3})$ but $R_1$ is red triangle connected to the large TCTF in $(B_1, B_2, B_{\geq 3})$. Since $B_1$ is not blue triangle connected to the large TCTF in $(R_1, R_2, R_{\geq 3})$ and because $(S_{B_1}, R_1')$ is $ht$-blue, we have that $(S_{B_1}, R_2)$ and $(S_{B_1}, R_{\geq 3})$ are $ht$-red. Now, since $R_1$ is red triangle connected to the large TCTF in $(B_1, B_2, B_{\geq 3})$ we have that $R_2$ and $R_{\geq 3}$ are not, because $(S_{B_1}, R_2)$ and $(S_{B_1}, R_{\geq 3})$ are $ht$-red this implies that $(B_2, R_2), (B_{\geq 3}, R_2)$ and  $(B_2, R_{\geq 3}), (B_{\geq 3}, R_{\geq 3})$ are $ht$-blue, which is absurd because it implies that both $B_2$ and $B_{\geq 3}$ are connected to the large TCTF in $(R_1, R_2, R_{\geq 3})$.

\underline{Case D.3}: $B_1$ is not blue triangle connected to the large TCTF in $(R_1, R_2, R_{\geq 3})$ and  $R_1$ is not blue triangle connected to the large TCTF in $(B_1, B_2, B_{\geq 3})$. In which case we notice that the blue cliques in $B_1$ are not triangle connected to the large blue TCTF in $(R_1, R_2, R_{\geq 3})$ and similarly the red cliques in $R_1$ are not triangle connected to the large red TCTF in $(B_1, B_2, B_{\geq 3})$. In particular, this implies that $(S_{B_1}, R_{\geq 3})$ and $(S_{B_1}, R_2)$ are $ht$-red, because we have that $(S_{B_1}, R_1')$ is $ht$-blue and $(R_1', \cup_{i\geq 2}R_i)$ is $ht$-blue. Likewise, we have  that $(S_{R_1}, B_{\geq 3})$ and $(S_{R_1}, B_2)$ are $ht$-blue. But this leaves us in a contradiction, indeed, neither $B_2$ nor $B_{\geq 3}$ can be triangle connected to $(R_1, R_2, R_3)$. Since $(S_{R_1}, B_{\geq 3})$ is $ht$-blue this means that $(R_2, B_{\geq 3})$ and $(R_{\geq 3}, B_{\geq 3})$ are $ht$-red. This is enough to get a contradiction, since we have $(R_2, B_{\geq 3})$ and $(R_{\geq 3}, B_{\geq 3})$ are $ht$-red but also $(S_{B_1}, R_{\geq 3})$ and $(S_{B_1}, R_2)$ are $ht$-red.




\end{proof}

\section{The colours of edges}\label{sec:finproof}
In this section we complete the proof of Lemma~\ref{MainLemma}. We first deduce an approximate version, proving that $B_{\ge3}\cup R_{\ge3}$ cannot have much more than $t$ vertices (which implies all components have roughly the correct size) and that most edges in various pairs have the `correct' colour. We then prove Lemma~\ref{MainLemma} by arguing that any edges with the `wrong' colour lead to triangle components which are much larger than they should be. The following is our approximate version.

\begin{lemma}\label{approxmainlemma}
There exists $h_0>0$ such that for every $0<h<h_0$ there exists $\epsilon_0>0$ such that for all $0<\epsilon<\epsilon_0$ there exists $t_0$ such that for every $t>t_0$ we have the following. Let $G$ be a 2-edge coloured graph with $(9-\epsilon)t$ vertices and minimum degree at least $(9-2\epsilon)t$. Fix  a collection of red and blue cliques as in Setting \ref{mainsetforG} with parameters $\epsilon$ and $t$, and define $B_1,B_2,\dots$ as in Setting~\ref{mainsetforG}. Then it holds:
\begin{itemize}
    \item $(2-h)t\leq \vass{B_1}, \vass{B_2}, \vass{R_1}, \vass{R_2}\leq (2+h)t$,
    \item $(1-h)t\leq \vass{B_{\geq 3}\cup R_{\geq 3}}\leq (1+h)t$,
    \item $\vass{G\setminus \cup_i (B_i\cup R_i)}\leq ht$,
    \item The following pairs are $h^2t$-blue: $(B_1,R_1)$, $(B_2,R_2)$, $(R_1,R_2)$, $(R_1, B_{\geq 3}\cup R_{\geq 3})$ and $(R_2, B_{\geq 3} \cup R_{\geq 3})$,
    \item The following pairs are $h^2t$-red: $(B_1,B_2)$, $(B_1, B_{\geq 3} \cup R_{\geq 3})$, $(B_2, B_{\geq 3} \cup R_{\geq 3})$, $(B_1,R_2)$ and $(B_2,R_1)$.
\end{itemize}
\end{lemma}
\begin{proof}
By Lemma \ref{lemmalastupperbound} we know that for $\epsilon>0$ small enough we have $\vass{B_1}, \vass{B_2},\vass{R_1},\vass{R_2}\leq (2+h^{\frac{3}{2}})t$ and therefore we have $\vass{B_{\geq 3} \cup R_{\geq 3}}\geq (1-5h^{\frac{3}{2}})t$. Without loss of generality, let us assume $\vass{B_{\geq 3}}\geq \vass{R_{\geq 3}}$.

\begin{claim}
We have that $R_1$ and $R_2$ are not red triangle connected to $(B_1, B_2, B_{\geq 3})$. Moreover, without loss of generality, we have $(R_1, B_1), (R_1, B_{\geq 3})$ and $(R_2, B_2), (R_2, B_{\geq 3})$ are $h^2t$-blue.
\end{claim}
\begin{claimproof}
Notice that $\vass{B_{1}}\geq \frac{(1-5h^{\frac{3}{2}})t}{2}$. Let us consider first that $R_1$ and $R_2$ are not red triangle connected to $(B_1, B_2, B_{\geq 3})$. Indeed, assume this is not the case and we have $3\vass{B_{\geq 3}}+\vass{R_i}<(3+h)t$ for some $i\in \llb 1,2\rrb$. Then we have $\vass{R_1}+\vass{R_2}+\vass{B_{\geq 3}}<(3+h+2+h^{\frac{3}{2}})t-2\vass{B_{\geq 3}}<(4-3h)t$ which is clearly absurd because it implies $\vass{B_1}+\vass{B_2}+\vass{R_{\geq 3}}\geq (5+2h)t$.

 We now claim that there is an ordering $i,j,k$ of $\llb 1, 2, \geq 3\rrb$ such that $(R_1, B_i), (R_1, B_j)$ and $(R_2, B_k), (R_2, B_j)$ are $h^2t$-blue.
    Indeed, by Lemma \ref{lemma17} we know that up to removing at most $h^5t$ vertices from each of $R_1$ and $R_2$, every vertex in $R_1\cup R_2$ has many blue edges in at least two among $\llb B_1, B_2, B_{\geq 3}\rrb$. This means that we can partition (not in a unique way) almost all the vertices of $R_1$ among the sets $S^{R_1}_{B_h}$, where the vertices in $S^{R_1}_{B_h}$ have their red neighbour in $\cup_\ell B_\ell$ contained in $B_h$. We define similarly $S^{R_2}_{B_h}$. We claim that just one of the $S^{R_1}_{B_h}$ is not trivial.
    
    Assume by contradiction that $S^{R_1}_{B_i}$ and $S^{R_1}_{B_j}$ have size at least $ht$. We cannot have that $S^{R_2}_{B_i}$ or $S^{R_2}_{B_j}$ have size at least $ht$, because otherwise we would have that $B_j$ and $B_k$ or $B_i$ and $B_k$ are connected respectively. Therefore we must have that $S^{R_2}_{B_k}$ contains almost all the vertices of $R_2$ and in particular is not trivial.
    Therefore we have that $S^{R_1}_{B_i}$, $S^{R_1}_{B_j}$ and $S^{R_2}_{B_k}$ are not trivial, which gives us that both $B_i$ and $B_j$ are in the same triangle-connected component. This implies that just one of the $S_{B_h}^{R_1}$ is nontrivial, and by symmetry the same is true for $R_2$. Moreover, we have that $S_{B_{\geq 3}}^{R_i}$ is trivial, because otherwise we would find a large blue TCTF in $(B_1, B_2, R_i\cup B_{\geq3})$.
    
    Finally, since by Lemma \ref{lemma19} we cannot have that $R_1$ and $R_2$ are $h^2t$-blue to the same pair, we know that each of $R_1$ and $R_2$ is $h^2t$-blue to $B_{\geq 3}$ and one between $B_1$ and $B_2$. We are going to assume without loss of generality that $(R_1, B_1)$ and $(R_1, B_{\geq 3})$ are $h^2t$-blue, and that $(R_2, B_2)$ and $(R_2, B_{\geq 3})$ are $h^2t$-blue, as we wanted.

\end{claimproof}

By the claim, we have that $(R_1, B_1)$, $(R_1, B_{\geq 3})$ and $(R_2, B_2)$, $(R_2, B_{\geq 3})$ are $h^2t$-blue. In particular, this means that we can find a blue TCTF in $(R_1, R_2, B_{\geq 3}\cup R_{\geq 3})$. This gives us immediately that $\vass{B_{\geq 3}\cup R_{\geq 3}}\leq (1+h^{\frac{3}{2}})t$ and in particular $\vass{B_1},\vass{B_2},\vass{R_1},\vass{R_2}\geq (2-h)t$.

Also, we get that $(B_1, R_2)$ and $(B_2, R_1)$ are $h^2t$-red. This holds because otherwise we would have both $B_{\geq 3}$ and $B_{2}$ are in the same connected component, indeed, $(B_{\geq 3}, R_1)$, $(B_{\geq 3}, R_2)$ and $(B_2, R_2)$ are $h^2t$ blue.

Assume now that $\vass{R_{\geq 3}}\geq h^{\tfrac{3}{2}t}$. We have that $(R_1, B_1), (R_1, B_{\geq 3})$ and $(R_2, B_2), (R_2, B_{\geq 3})$ are $h^2t$-blue, this gives us that $B_{\geq 3}$ is blue triangle connected to $(R_1, R_2, R_{\geq 3})$ (which is a non-trivial TCTF) which in turn gives us that $B_1$ and $B_2$ are not. From this last fact we can conclude that $(B_1, R_{\geq 3})$ and $(B_2, R_{\geq 3})$ are all $h^2t$-red.


So we have the construction that we wanted up to change the indices between $B_1$, $B_2$ and $R_1$, $R_2$ respectively.
\end{proof}

Let us now prove Lemma \ref{MainLemma} that we restate for convenience.

\MainLem*

\begin{proof}[Proof of Lemma~\ref{MainLemma}]
We are going to refine Lemma \ref{approxmainlemma} to obtain an exact result.

By Lemma \ref{approxmainlemma} we have that there exists $\delta_0>0$ such that for $\delta_0>h,\lambda>0$ there exists $\epsilon_0>0$ and $t_0\in\mathbb{N}$ such that for any $t>t_0$ and $\epsilon_0>\epsilon>0$ if $G$ is a 2-edge-coloured graph over $(9-\epsilon)t$ vertices with minimum degree at least $(9-2\epsilon)$ and without a monochromatic TCTF on at least $3(1+\epsilon)t$ vertices, then we can partition $V(G)$ in the sets $B_1, B_2, R_1, R_2, Z, T$ (where the $B_i$ and $R_i$ are as in Setting \ref{mainsetforG} and where where $Z=B_{\geq 3}\cup R_{\geq 3}$ and $T$ is the set of vertices which are not already counted) such that the following holds:
\begin{itemize}
    \item $(2-h)t\leq \vass{B_1},\vass{B_2},\vass{R_1},\vass{R_2}\leq (2+h)t$,
    \item $(1-h)t\leq \vass{Z}\leq (1+h)t$,
    \item $\vass{T}\leq ht$,
    \item The following pairs are $\lambda t$-blue: $(B_1,R_1)$, $(B_2,R_2)$, $(R_1,R_2)$, $(R_1, Z)$ and $(R_2, Z)$,
    \item The following pairs are $\lambda t$-red: $(B_1,B_2)$, $(B_1, Z)$, $(B_2, Z)$, $(B_1,R_2)$ and $(B_2,R_1)$.
\end{itemize}

We first need to slightly prune our sets. We start by removing from $B_1$ (and putting in $T$) the vertices with more than $\frac{1}{8}\lambda$ red neighbours to $R_1$ and the vertices with more than $\lambda$ blue neighbours to either $B_2, R_2$ or $Z$. We do the same to $B_2$, $R_1$ and $R_2$ accordingly to the colour of the pairs we are considering.

Up to reducing $\epsilon_0$, we are still respecting all the bounds on the sizes that we need for Lemma \ref{MainLemma}, but we have a slightly better result on the state of the ``problematic'' edges. Indeed, we know that there are no vertices outside $T$ that witness more than $\lambda$ ``problematic'' edges.

We now just need to prove that $G[B_1]$, $G[B_2]$, $G[R_1]$, $G[R_2]$ and  $(B_1,R_1)$, $(B_2,R_2)$, $(R_1, Z)$, $(R_2, Z)$, $(B_1, Z)$, $(B_2, Z)$, $(B_1,R_2)$, $(B_2,R_1)$ are entirely monochromatic.

The proof to show that $G[B_1]$, $G[B_2]$, $G[R_1]$, $G[R_2]$ are monochromatic have the same structure. Therefore without loss of generality we show that $G[B_1]$ is entirely blue. Assume by contradiction that we can find $u,v$ in $B_1$ such that $uv$ is red. By our earlier pruning we know that both $u$ and $v$ have at most $\frac{1}{8}\lambda$ blue neighbours in $R_2$. Therefore  $uv$ is triangle connected to one of the red cliques of $R_2$ (and therefore to all red cliques of $R_2$). Let us now prove that $uv$ is also triangle connected to the large red TCTF in $(B_1, B_2, Z)$ (which is enough to conclude since we would then be able to find a large triangle-connected triangle component). Almost all the red edges in $(\llb u\rrb, B_2)$ are triangle connected to $uv$, indeed, all but at most $\frac{1}{8}\lambda$ of them are in a red triangle with $uv$, the same holds for the red edges in $(\llb u\rrb, Z)$. This means that there are at most $\lambda$ vertices in either $B_2$ or $Z$ that witness a red edge in $(B_2, Z)$ which is not triangle connected to $uv$. But this is absurd, as we mentioned before, since it implies that a large red TCTF in $(B_1, B_2, Z)$ is triangle connected to $uv$.

Let us now prove that $(B_1,R_1)$, $(B_2,R_2)$, $(R_1, Z)$, $(R_2, Z)$, $(B_1, Z)$, $(B_2, Z)$, $(B_1,R_2)$, $(B_2,R_1)$ are entirely monochromatic. The structure of these proofs is always the same, so without loss of generality we prove that $(B_1, R_1)$ is monochromatic. Assume it is not, and let $uv$ be a red edge between $B_1$ and $R_1$ (with $u\in B_1$). We prove that $uv$ is triangle connected both to one clique of $R_1$ (and therefore all cliques of $R_1$) and to the large red triangle-connected component in $(B_1, B_2, Z)$, which is absurd since this would give a large red TCTF. 

We first show that $uv$ is triangle connected to $R_2$, let $w_1\in R_2$ such that $vw_1$ is an edge (which has to be red by our previous proof that $G[R_1]$ is entirely red). Then by our pruning we know that $u, v$ and $w_1$ share a red neighbour in $B_2$.
We now observe that if $w_2w_3$ is a red edge between $B_2$ and $Z$ (with $w_2\in B_2$) such that $w_2$ is a red neighbour of both $u$ and $v$ and $w_3$ is a red neighbour of $w_2$ and $u$, then $w_2w_3$ is triangle connected to $uv$. By the pruning we did earlier, we can say that most of the red edges between $B_2$ and $Z$ are triangle connected to $uv$, which is what we wanted.

Up to changing the roles of the clusters, the other proofs have the same structure.
\end{proof}

\section{Regularity Method: proofs of Lemma~\ref{regularityresult} and Theorem~\ref{thm:boundBW}}\label{sec:reg}

In this section we state the Regularity Lemma and Blow-up Lemma, and use them to deduce Lemma~\ref{regularityresult} and Theorem~\ref{thm:boundBW} from Lemma~\ref{MainLemma}.

\begin{definition}[density, $\eps$-regular]
Let $G$ be a graph and let $X, Y$ be disjoint subsets in $V(G)$. We define the \emph{density} $d(X,Y)$ between $X$ and $Y$ to be:
$$d(X,Y):=\frac{e(X,Y)}{\vass{X}\vass{Y}}\,.$$
Given $\epsilon>0$, we say that $(X,Y)$ is $\epsilon$-regular if for every $X'\subseteq X,\ Y'\subseteq Y$ such that $\vass{X'}>\epsilon \vass{X}$ and $\vass{Y'}>\epsilon\vass{Y}$ we have
$\vass{d(X',Y')-d(X,Y)}<\epsilon$.
\end{definition}

We use the following version of the Regularity Lemma.  We will apply this to the graph of red edges within $K_n$, and observe that if $(X,Y)$ is $\eps$-regular in red edges then, since the blue edges are the complement of the red edges, it is also $\eps$-regular in blue.

\begin{lemma}[Regularity Lemma]\label{RegularityLemma}
For every $\epsilon \in (0,1)$ there are $M, N_0\in\mathbb{N}$ such that the following holds. Let $G$ be a graph on $n \geq N_0$ vertices, then there is a partition $\llb V_0,\dots{}, V_m\rrb$ of $V(G)$ such that the following conditions hold. We have $|V_0|\le\eps^{-1}$ and $\eps^{-1}\le m\le M$. We have $|V_1|=\dots=|V_m|$. Finally, for any given $i\in[m]$, for all but at most $\eps m$ choices of $j\in[m]$ the pair $(V_i,V_j)$ is $\eps$-regular in $G$.
\end{lemma}
This version follows from the original version of Szemer\'edi~\cite{SzeReg} (which is similar but bounds the total number of irregular pairs by $\eps m^2$ rather than the number of irregular pairs meeting a part) applied with parameter $\tfrac18\eps^2$, followed by removing parts incident to more than $\tfrac12\eps m$ irregular pairs (of which there are at most $\tfrac12\eps m$) to $V_0$; we leave the details to the reader.

Given $\eps,d>0$, a $2$-edge-coloured complete graph $G$ and a partition obtained by applying Lemma~\ref{RegularityLemma} with parameter $\eps$ to the subgraph of red edges, we define the \emph{$(\eps,d)$-reduced graph} of $G$ (with respect to the partition) to be the graph $H$ on vertex set $[m]$, in which an edge $ij$ is present if it is $\epsilon$-regular, and assigned the colour red if its density in red is at least $1-d$, blue if its density in blue is at least $1-d$, and otherwise purple. 

We will see that for the purposes of embedding a graph into $G$, we can treat purple edges as being either red or blue as we desire, so that a large TCTF in (red $\cup$ purple) edges, or in (blue $\cup$ purple) edges in the reduced graph implies the existence of the square paths and cycles in $G$ we need. In order to apply Lemma~\ref{MainLemma} in this setting, we deduce the following consequence, which roughly says that either we are done or we get essentially the same partition as in Lemma~\ref{MainLemma}, in particular there are very few purple edges.

\begin{lemma}\label{MainLemmaWithPurple}
For any $\delta>0$ there exists $\eps>0$  such that for any $t\geq \tfrac{1}{\eps}$, if $G$ is a $\llb\text{red, blue, purple}\rrb$-edge-coloured graph on $(9-\eps)t$ vertices with minimum degree at least $(9-2\eps)t$, then either there is a choice of a colour between blue and red such that if we colour all the purple edges of that colour we can find a monochromatic TCTF on at least $3(1+\eps)t$ vertices in $G$ or $V(G)$ can be partitioned in sets $\llb B_1, B_2, R_1, R_2, Z, T\rrb$ such that the following hold.
\begin{enumerate}[label=(\alph*)]
    \item\label{MLP:i} $(2-\delta)t\leq \vass{B_1}, \vass{B_2}, \vass{R_1}, \vass{R_2}\leq (2+\delta)t$,
    \item\label{MLP:ii} $(1-\delta)t\leq \vass{Z}\leq (1+\delta)t$,
    \item\label{MLP:iii} all the edges in $G[B_1]$ and $G[B_2]$ are blue, and all the edges in $G[R_1]$ and $G[R_2]$ are red,
    \item\label{MLP:iv} the pairs $(B_1,R_1)$, $(B_2, R_2)$, $(R_1, Z)$ and $(R_2,Z)$ are entirely blue. Moreover, the pairs $(B_1,R_2)$, $(B_2, R_1)$, $(B_1, Z)$ and $(B_2,Z)$ are entirely red,
    \item\label{MLP:v} the pair $(B_1, B_2)$ is $\delta t$-red, while the pair $(R_1, R_2)$ is $\delta t$-blue, and
    \item\label{MLP:vi} $\vass{T}\leq \delta t$.
\end{enumerate}
\end{lemma}
\begin{proof}
Let $\epsilon$ be given by Lemma~\ref{MainLemma} for input $h=\lambda=\tfrac1{1000}\delta$; without loss of generality we may suppose $\delta$ is sufficiently small for this application.

Let $G$ be a coloured graph satisfying the conditions of the lemma, and suppose there is neither a red-purple TCTF over $3(1+\eps)t$ vertices nor a blue-purple TCTF over $3(1+\eps)t$ vertices.

Let $G^r$ be the graph obtained from $G$ by recolouring the purple edges red, and similarly $G^b$ by recolouring them blue. Let $R_1^r,R_2^r,B_1^r,B_2^r,X^r,T^r$ be the partition obtained by applying Lemma~\ref{MainLemma} to $G^r$, and define similarly the partition for $G^b$ replacing $r$ with $b$. Observe that if we swap $R^1_r$ and $R^2_r$, and also $B^1_r$ and $B_2^r$, we still have a partition satisfying the conclusion of Lemma~\ref{MainLemma}. If $\big|R_2^r\cap R_1^b\big|>\big|R_1^r\cap R^1_b\big|$, we perform this swap (and in an abuse of notation continue to use the same letters for the swapped classes).

We define $R_i:=R_i^r\cap R_i^b$ and $B_i:=B_i^r\cap B_i^b$ for each $i=1,2$, and $Z:=Z_r\cap Z_b$ and finally $T:=V(G)\setminus(B_1\cup B_2\cup R_1\cup R_2\cup Z)$. We will now prove this partition satisfies the conclusions of the lemma. Observe that the statements in~\ref{MLP:iii},~\ref{MLP:iv} and~\ref{MLP:v} about sets or pairs being entirely red, or $\delta t$-red, follow directly from the same statements for the partition of $G^b$, and the corresponding ones about being blue from the partition of $G^r$; what remains is to prove these sets have the correct sizes.

To begin with, observe that all edges in $R_1^b$ are red in $G^b$ and so also in $G$. It follows that $R_1^b$ intersects $B_i^r$ in at most one vertex for each $i=1,2$, since otherwise $B_i^r$ would contain a red edge. Thus $R_1^b$ has at least $(2-\frac{1}{100}\delta)t-2$ vertices which are not in $B_1^r\cap B_2^r$. These vertices cannot all be in $T^r\cup Z^r$, which is too small, so $R_1^b$ has a vertex in at least one of $R_1^r$ and $R_2^r$. Now $R^1_b$ cannot have vertices in $Z^r$, since all edges from $Z^r$ to $R_1^r\cup R_2^r$ are not red. It follows that all but at most $\tfrac1{1000}\delta t+2$ vertices of $R^1_b$ are in $R_1^r\cup R_2^r$, and by the observation above there are at least as many vertices in $R_1^r$ as $R_2^r$. Since $(R_1^r,R_2^r)$ is $\tfrac1{1000}\delta t$-blue, and all edges in $R^1_b$ are red, we see $R^1_b$ has at most $\tfrac1{1000}\delta t$ vertices in $R_2^r$. Finally, we conclude $|R_1|\ge(2-\tfrac1{100}\delta)t$. We also have $|R_1|\le |R_1^r|\le(2+\tfrac1{1000}\delta)t$. By a similar argument (noting that $R_1^b$ and $R_2^b$ are disjoint) we obtain
\[(2-\tfrac{1}{100}\delta)t\le |R_i|\le (2+\tfrac1{1000}\delta)t\]
for each $i=1,2$.

We make a similar argument for $B_1^r$. As above, we can conclude that all but at most $\tfrac1{1000}\delta)t+2$ vertices of $B_1^r$ are in $B_1^b\cup B_2^b$. However we can now observe that all edges from $B_1^r$ to $R_1\subset R_1^r$ are blue, while the edges from $B_2^b$ to $R_1\subset R_1^b$ are red. It follows that $B_1^r$ is disjoint from $B_2^b$, and we obtain
\[(2-\tfrac{1}{100}\delta)t\le |B_i|\le (2+\tfrac1{1000}\delta)t\]
for $i=1$, and, by the similar argument, for $i=2$.

Now $Z^r$ and $Z^b$ are two sets of size at least $(1-\tfrac1{1000}\delta)t$ in $V(G)\setminus(R_1\cup R_2\cup B_1\cup B_2)$, which has size at most $(9-\eps)t-4(2-\tfrac1{100}\delta)t\le t+\tfrac2{50}\delta t$. It follows their intersection $Z$ has size at least $(1-\tfrac{1}{10}\delta)t$, and at most $|Z^r|\le(1+\tfrac1{1000}\delta)t$. Finally, putting these size bounds together we have~\ref{MLP:i},~\ref{MLP:ii} and an upper bound on $|T|$ giving~\ref{MLP:vi}.
\end{proof}

To go with the above lemma, we state the following two embedding lemmas. The first is a corollary of~\cite[Lemma~7.1]{ABHKP}, though one could use the original Blow-up Lemma of Koml\'os, S\'ark\"ozy and Szemer\'edi~\cite{KSS} with some extra technical work in the proof of Theorem~\ref{thm:boundBW}. To deduce the following statement from~\cite[Lemma~7.1]{ABHKP}, we take $R'$ to be the graph with zero edges and $\Delta_{R'}=1$, we take $\kappa=2$, and we add to $H$ for each $i\in R$ a set of $|V_i|-|\phi^{-1}(i)|$ isolated vertices which (extending $\phi$) we map to $i$ and let be the buffer vertices $\tilde{X}_i$. 

\begin{theorem}\label{thm:blowup}
Given $d,\gamma>0$ and $\Delta\in\mathbb{N}$, there exists $\eps>0$ such that for any given $T$, the following holds for all $m\ge m_0$. Let $R$ be any graph on $[t]$, where $t\le T$. Let $V_1,\dots, V_t$ be disjoint vertex sets with $m\le |V_i|\le 2|V_j|$ for each $i,j\in[t]$, and suppose that $G$ is a graph on $V_1\cup\dots\cup V_t$ such that $(V_i,V_j)$ is an $\eps$-regular pair of density at least $d$ for each $ij\in E(R)$. Suppose that $H$ is any graph with $\Delta(H)\le\Delta$ such that there exists a graph homomorphism $\phi:H\to R$ satisfying $\big|\phi^{-1}(i)\big|\le(1-\gamma)|V_i|$. Then $H$ is a subgraph of $G$.
\end{theorem}

The next is a consequence of the (original) Blow-up Lemma derived in~\cite{AllenBoettcherHladky2011}.

\begin{lemma}[Embedding Lemma, Allen, B\"{o}ttcher, Hladk\'{y}~\cite{AllenBoettcherHladky2011}]\label{EmbeddingLemma}
For all $d>0$ there exists $\epsilon_{EL}>0$ with the following property. Given $0<\epsilon<\epsilon_{EL}$, for every $m_{EL}\in\mathbb{N}$ there exists $n_{EL}\in \mathbb{N}$ such that the following holds for any graph $G$ on $n>n_{EL}$ vertices with $(\epsilon, d)$-reduced graph $R$ on $m\leq m_{EL}$ vertices. 
Let $\xi(R)$ be the size of the largest TCTF in $R$, then for every $\ell\in \mathbb{N}$ with $3\ell\leq (1-d)\xi(R)\tfrac{n}{m}$ we have $C_{3\ell}^2\subset G$.
\end{lemma}

We are now in a position to prove Lemma~\ref{regularityresult}, which we restate for convenience.

\RegResult*

\begin{proof}
Given $\alpha>0$, let $d$ be such that $10000\alpha^{-2}d$ is returned by Lemma~\ref{MainLemmaWithPurple} when we use as input $\tfrac{\alpha^2}{10000}$. Let $\epsilon_{EL}$ be returned by Lemma~\ref{EmbeddingLemma} for input $d$, and let $\eps=\min(\tfrac1{10}d,\eps_{EL},\tfrac1{10000}\alpha^2)$.
Let now $N_0$ and $M$ be returned by Lemma \ref{RegularityLemma} with input $\eps$. We let $n_{EL}$ be returned by Lemma \ref{EmbeddingLemma} for input $d$, $\eps$ and $m_{EL}=M$. Finally, let $\delta=d$ and $n_0=\max(100\eps^{-1}, N_0,N_1)$ be the constants returned by the lemma..

Let us now fix some $n>n_0$ and, for $N>(9-\delta)n$, a $2$-edge-colouring $G$ of $K_N$.

We apply Lemma \ref{RegularityLemma} with parameter as above to the red subgraph of $G$ to get a partition $V_0,\dots{}, V_m$ of $V(G)$, with $\eps^{-1}\le m\le M$, as in Lemma~\ref{RegularityLemma}. 
Let $H$ be the $(\eps,d)$-reduced graph of $G$. Since any cluster $V_i$ is in at most $\eps m$ irregular pairs, we have  $\delta(H)\geq (1-\eps)m-1$. Let $t=\tfrac{m}{9-10\alpha^{-1}d}$, so that $H$ has $(9-10\alpha^{-1}d)t$ vertices and, by choice of $\eps$, minimum degree at least $(9-20\alpha^{-1}d)t$. By Lemma~\ref{MainLemmaWithPurple} with constants as above, one of the following occurs.

It could be that $H$ contains a red-purple TCTF over $3(1+10d)t=\tfrac{3(1+10d)}{9-10\alpha^{-1}d}m\ge\tfrac13(1+10d)m$ vertices. Applying Lemma~\ref{EmbeddingLemma} with constants as above, we conclude that $G$ contains a red $C^2_{3s}$ for each $s\le (1-d)\cdot \tfrac13(1+10d)\cdot(9-d)n\ge 3(1+d)n$. But then in particular $G$ contains a red copy of $C_{3n}^2$ and $P_{3n+2}^2$ and we are done. Similarly, if $H$ contains a blue-purple TCTF over $3(1+10d)t$ vertices then $G$ contains a blue copy of $C_{3n}^2$ and $P_{3n+2}^2$ and we are done.

Alternatively, by Lemma \ref{MainLemmaWithPurple} we get a partition of $V(H)$ in sets $B_1$, $B_2$, $R_1$, $R_2$, $Z''$ and $T$. We obtain from this a partition of $V(G)$, setting $X'_j=\bigcup_{i\in B_j}V_i$ and $Y'_j=\bigcup_{i\in R_j}V_i$ for $j=1,2$, setting $Z':=\bigcup_{i\in Z''}V_i$, and letting $R'$ be the remaining vertices. Since we applied Lemma~\ref{MainLemmaWithPurple} with input $\tfrac{\alpha^2}{10000}$ and by choice of $\eps$, we have properties~\ref{conditionA} and~\ref{conditionB} of Lemma~\ref{approxmainlemma} with $\tfrac{1}{1000}\alpha^2$ instead of $\alpha$.

Since $(B_1,B_2)$ is $\tfrac1{10000}\alpha^2 t$-red, the number of blue edges in $G$ between $X'_1$ and $X'_2$ is at most
\[d|X'_1||X'_2|+\tfrac1{1000}\alpha^2 n|X'_2|+\tfrac{1}{1000}\alpha^2 n|X'_1|\le\tfrac{1}{200}\alpha^2 n^2 \,,\]
where the inequality uses $d\le\tfrac{1}{10000}\alpha^2$. In particular there are less than $\tfrac1{200}\alpha n$ vertices in $X'_1$ which have more than $\alpha n$ blue neighbours in $X'_2$, and similarly swapping $X'_1$ and $X'_2$. By the identical calculation, an analogous statement holds for $Y'_1$ and $Y'_2$ in red.

We now claim that at most $\eps|X'_1|$ vertices in $X'_1$ send $\alpha n$ or more red edges to $Y'_1$. Suppose for a contradiction this statement is false. By averaging, there is a cluster $V_i$ with $i\in B_1$ such that a set $S$ of $\eps|V_i|$ vertices of $V_i$ all send $\alpha n$ or more red edges to $Y'_1$. Since $V_i$ is in irregular pairs with at most $\eps m$ other clusters, at most $2\eps n$ red edges from each $s\in S$ go to clusters of $R_1$ that make irregular pairs with $V_i$. The remaining at least $\tfrac12\alpha n|S|$ edges from $S$ therefore go to the remaining less than $3m$ clusters $V_j$ with $j\in R_1$, which all form $\eps$-regular pairs with $V_i$ that have density at most $d$ in red. Again by averaging, there is a cluster $V_j$ with $j\in R_1$ such that $(V_i,V_j)$ is $\eps$-regular and has red density at most $d$, but also receives at least $\tfrac{\alpha n|S|}{6m}>2d|V_j||S|$ red edges from $S$. But this, since $\eps<d$ and $|S|\ge\eps|V_i|$, is a contradiction to regularity of $(V_i,V_j)$.

By a similar argument, at most $\eps|X'_i|$ vertices in $X'_i$ send edges of the `wrong' colour to each $Y'_j$ or to $Z'$ or vice versa. We can modify the argument slightly to show that at most $\eps|X'_1|$ vertices of $X'_1$ have more than $\alpha n$ red neighbours in $X'_1$: again we can find a set $S$ in a cluster $V_i$ with $i\in B'_1$ whose vertices all have more than $\alpha n$ red neighbours in $X'_1$, but we need to discard both red edges in irregular pairs at $V_i$ and also edges within $V_i$. Since $|V_i|\le\tfrac{m}{n}\le\eps n$, there are in total at most $2\eps n$ such edges, which is the same bound we used above and from this point the proof above works as written.

We now let $X_1$ be obtained from $X'_1$ by discarding all vertices which have more than $\alpha n$ edges of the `wrong' colour to any of $X'_i$ or $Y'_i$ or $Z'$. By the above calculations, in total we discard at most $4\eps |X'_1|+\tfrac1{200}\alpha n\le\tfrac1{100}\alpha n$ vertices of $X'_1$. We define similarly $X_2,Y_1,Y_2,Z$, and similarly remove at most $\tfrac1{100}\alpha n$ vertices in each case. Finally, we let $R$ denote the set of all vertices not in $X_1\cup X_2\cup Y_1\cup Y_2\cup Z$. By construction, each $X_i$, $Y_i$ and $Z$ has the claimed size; and $|R|\le\alpha n$ follows since each vertex of $R$ was either in $V_0$, or $V_i$ for some $i\in T$, or removed from $X'_i$ or $Y'_i$ or $Z'$. There are at most $\eps n+\tfrac{\alpha^2}{10000}n+5\cdot\tfrac1{100}\alpha n$ such.

Finally, by definition each vertex of $X_1$ has at most $\alpha n$ edges of the `wrong' colour to any of $X'_i,Y'_i$ or $Z'$, which are supersets of $X_i,Y_i,Z$ respectively, giving the required bounds on `wrong' coloured edges at $X_1$. The same holds for $X_2,Y_1,Y_2,Z$ by the similar argument.
\end{proof}

Finally, we prove Theorem~\ref{thm:boundBW}. First, we deduce from Lemma~\ref{MainLemma} that if $G$ satisfies the conditions of Theorem~\ref{thm:boundBW}, then the reduced graph $R$ of $G$ is an $m$-vertex graph which contains a monochromatic TCTF on nearly $\tfrac13m$ vertices. Suppose this is red. We then show how to construct a homomorphism from any given $H$ satisfying the conditions of Theorem~\ref{thm:boundBW} to the red subgraph of $R$ which does not overload any vertex $i$ of $R$, i.e.\ map too many vertices to $i$, and finally apply Theorem~\ref{thm:blowup} to find the desired monochromatic copy of $H$ in $G$.

The only tricky step of this sketch is to construct the required homomorphism. We split up $V(H)$ into \emph{chunks} and \emph{fragments}, which are intervals in the bandwidth ordering, alternating between chunks and fragments. Each fragment is of equal length and their total size is tiny compared to the size of a cluster, and the chunks are of equal length and much bigger than the fragments (but still much smaller than the size of a cluster). Given our TCTF in $R$, we put an order $T_1,\dots,T_k$ (arbitrarily) on the triangles of the TCTF, and fix for each $1\le i\le k-1$ a tight walk from $T_i$ to $T_{i+1}$. We assign each chunk of $H$ to some $T_i$ where $i$ is chosen uniformly and independently from $[k]$. We claim that it is possible to now construct a homomorphism where each chunk will be mapped entirely to its assigned triangle, using the fragments to connect up along the fixed tight walks, and that this homomorphism will with positive probability not overload any vertex of $R$: the point here will be to analyse the assignment of chunks, since the total size of all fragments is tiny.

\begin{proof}[Proof of Theorem~\ref{thm:boundBW}]
Given $\gamma>0$ and $\Delta$, we fix $h\le\tfrac{\gamma}{1000}$ and $\lambda>0$ (which will play no further r\^ole in this proof) sufficiently small for Lemma~\ref{MainLemma}, and let $2\eps'$ and $t_0$ be the returned constants. We let $\eps>0$ be returned by Theorem~\ref{thm:blowup} for input $d=\tfrac12$, $\tfrac{\gamma}{100}$ and $\Delta$. Without loss of generality, we may presume $\eps<\tfrac1{10}\min(t_0^{-1},\eps',\gamma)$. We input $\eps$ and $d=\tfrac12$ to Theorem~\ref{RegularityLemma} and let $M,N_0$ be the returned constants. We input $T=M$ to Theorem~\ref{thm:blowup}, with the other parameters as above, and choose $N_1$ such that the returned constant $m_0\le N_1/M$.

We set $\rho=\tfrac{1}{60000}M^{-3}\gamma^2$ and $\beta=\tfrac{1}{100}M^{-4}\gamma$.

Suppose now $n\ge\max(N_0,N_1)$. 

Let $N=(9+\gamma)n$. Given a $2$-edge-coloured $K_N$, we apply Lemma~\ref{RegularityLemma} with constants as above, to the graph of red edges in $K_N$, to obtain a partition $V(K_N)=V_0\cup V_1\cup\dots\cup V_m$, where $\eps^{-1}\le m\le M$. By construction, the number of vertices in each part $V_i$ with $1\le i\le m$ is at least $\tfrac{(9+\gamma/2)n}{m}$.

Let $R$ be the corresponding coloured reduced graph on $[m]$, in which we colour a pair $ij$ red if $(V_i,V_j)$ is $\eps$-regular and has density in red at least $\tfrac12$, blue if it is $\eps$-regular and has density in blue strictly bigger than $\tfrac12$, and otherwise (i.e.\ if the pair is irregular) we do not put an edge $ij$. By construction, we have $\delta(R)\ge (1-\eps)m$.

Let $t=m/(9-\eps')$, so that $R$ has $(9-\eps')t$ vertices and minimum degree at least $(9-2\eps')t$. By Lemma~\ref{MainLemma}, either $R$ contains a monochromatic $3(1+\eps')t$-vertex TCTF, or we obtain a partition of $V(R)$ as described in that lemma. In particular, there is a set $B_1$ of at least $(2-h)t$ vertices and a disjoint set $R_1$ of at least $(2-h)t$ vertices, such that any triangle with two vertices in $B_1$ and one in $R_1$ is monochromatic blue (and so all such triangles are in a blue triangle component). It follows that choosing $(1-h)t$ disjoint such triangles greedily we obtain a monochromatic TCTF with $3(1-h)t$ vertices. We see that in all cases $R$ contains a monochromatic TCTF on at least $3(1-h)t\ge\tfrac13(1-h)m=:3k$ vertices.

Fix such a TCTF, let its triangles be $T_1,\dots,T_k$ and suppose without loss of generality that it is red. By definition of red triangle-connectedness, for each $1\le i\le k-1$ there is a red triangle walk in $R$ from $T_i$ to $T_{i+1}$, and we fix for each $i$ one such $W_i$ chosen to be of minimum length. Thus $W_i$ is a sequence of triangles, starting with $T_i$ and ending with $T_{i+1}$, in which each pair of consecutive triangles shares two vertices. Finally, we assign labels $1,2,3$ to the vertices of all these triangles as follows: we label the vertices of $T_1$ in an arbitrary order, then assign labels to the successive triangles of $W_1,W_2,\dots,W_{k-1}$ in order as follows: when we assign labels to the next triangle, we keep the labels of the two vertices it shares with the previous triangle, and give the missing label to the third vertex. Note that a given vertex, or a given edge, might receive different labellings for different triangles, and indeed if a triangle appears in two different walks it might receive different labellings in the different walks.

Let $H$ be a graph with maximum degree at most $\Delta$, bandwidth at most $\beta n$, and a fixed $3$-vertex colouring in which no colour class has more than $n$ vertices. We split $V(H)$ into consecutive intervals $C_1,F_1,C_2,F_2,\dots,F_{s-1},C_s$ as follows: we let each $C_i$ (except perhaps the last two, which can be of any size) consist of $\rho n$ vertices, and each $F_i$ be of size $M^2\beta n$. For each $1\le i\le s$ we pick $\pi(i)\in[k]$ uniformly and independently at random. We now define a homomorphism $\psi:H\to R$ as follows. If $x$ is a vertex of the chunk $C_i$ for some $i$, and its colour in the fixed $3$-colouring of $H$ is $j\in[3]$, then we set $\psi(x)$ equal to the vertex of $T_{\pi(i)}$ with label $j$. We now describe how to construct $\psi$ on the fragment $F_1$; the same procedure is used for each subsequent fragment with the obvious updates. We separate $F_1$ into intervals of length $\beta n$. If $x$ is in the $i$th interval, and has colour $j$ in the $3$-colouring, then we set $\psi(x)$ equal to the vertex of the $i$th triangle after $T_1$ in $W_1$ with label $j$. If there is no such triangle (i.e.\ the walk has already reached $T_2$) then we set $\psi(x)$ equal to the vertex labelled $j$ in $T_2$. We claim that this last event occurs for the final interval. Indeed, if two triangles of $W_1$ both contain a given edge $e$ of $R$, then by minimality they are consecutive triangles in the walk, so $W_1$ has less than $M^2$ triangles.

We claim that this construction gives a homomorphism from $H$ to the red subgraph of $R$. Indeed, suppose $xy$ is an edge of $H$. Then $x$ and $y$ have different colours in the $3$-colouring, and they are separated by at most $\beta n$ in the bandwidth ordering. By construction, $x$ is assigned to a vertex of some triangle $T$ according to its colour. The vertex $y$ is assigned to a triangle $T'$ according to its colour; and either $T=T'$ or $T$ and $T'$ are consecutive triangles on one of the fixed walks, in particular they share two vertices and their labels are consistent on those two vertices. Either way, $x$ and $y$ are mapped to a red edge of $R$ (the only non-edge is if $T\neq T'$ and it goes between the two vertices of the symmetric difference of $T$ and $T'$, which both have the same label: but $x$ and $y$ have different colours).

We still need to justify that with high probability $\psi$ does not overload any vertex of $R$. To begin with, observe that the total number of vertices in the fragments is at most $M\cdot M^2\beta n=M^3\beta n\le\tfrac{\gamma n}{100m}$, which is much smaller than the size of any cluster $V_i$. In particular, if $i$ is not in any triangle of the TCTF, then $|\psi^{-1}(i)|<\tfrac12|V_i|$ as desired. Now consider the vertex $u$ of $T_i$ with label $j$. Apart from the at most $M^3\beta n$ vertices of the fragments, the vertices of $\psi^{-1}(u)$ are vertices of chunks with colour $j$. There are at most $n$ vertices in chunks of colour $j$ in total, and each such chunk has probability $1/k$ of being assigned to $T_i$. We see that the expected number of chunk vertices in $\psi^{-1}(u)$ is at most $n/k$. The probability that the actual number of such vertices exceeds $n/k$ by $s$ is by Hoeffding's inequality at most
\[\exp\big(-\tfrac{s^2}{2\cdot 3\rho^{-1}\cdot (\rho n)^2}\big)=\exp\big(-\tfrac{s^2}{6\rho n^2}\big)\,,\]
where we used that there are at most $3\rho^{-1}$ chunks, and the maximum contribution of a given chunk to $|\psi^{-1}(u)|$ is at most $\rho n$. Choosing $s=\tfrac{1}{100}\gamma M^{-1}n$, by choice of $\rho$ the probability that $|\psi^{-1}(u)|\ge n/k+s+M^3\beta n$ (the last term accounts for vertices in fragments) is at most $\exp(-M)$. In particular, with positive probability we have
\[|\psi^{-1}(u)|\le\tfrac{n}{k}+\tfrac{1}{100}\gamma\tfrac{n}{M}+M^3\beta n\]
for every $u\in V(R)$. Suppose this event occurs. Substituting our values for $\beta$, $k$ and finally $h$, we get
\[|\psi^{-1}(u)|\le\tfrac{9n}{(1-h)m}+\tfrac{1}{100}\gamma\tfrac{n}{m}+\tfrac{1}{100}\gamma\tfrac{n}{m}\le\tfrac{9n}{m}(1+2h)+\tfrac{1}{10}\gamma\tfrac{n}{m}\le\frac{(9+\tfrac{\gamma}{5})n}{m}\,.\]
Since $|V_u|\ge\tfrac{(9+\gamma/2)n}{m}$, as observed at the start of this proof, we have $|\psi^{-1}(u)|\le(1-\tfrac{\gamma}{100})|V_u|$ for every $u\in V(R)$. This is the required condition to apply Theorem~\ref{thm:blowup}.

Finally, by Theorem~\ref{thm:blowup} we conclude that there is a red copy of $H$ in the $2$-coloured $K_N$.
\end{proof}


\section{Proof of Theorem~\ref{thm:main}}\label{sec:exact}
We are now ready to prove the main result of this paper, which we restate for convenience. Recall that we established the lower bound in Section~\ref{sec:note}, and what remains is to prove the corresponding upper bound. We give the full details for the square of a path, the square of a cycle case is similar.

\thmmain*

\begin{proof}[Proof of Theorem~\ref{thm:main}, upper bound for $P_{3n+1}^2$]
Let $n_0$ and $\delta$ be given by Lemma \ref{regularityresult} when we set $\alpha=\frac{1}{1000}$ (we are not trying to optimise this value) and then let us fix $n>\max(n_0, \frac{3}{\delta})$ and $N\geq 9n-3$. Let now $G$ be any 2-edge-colouring of $K_{N}$. We suppose for a contradiction that $G$ does not contain a monochromatic $P^2_{3n+1}$.
By Lemma \ref{regularityresult}, since $G$ does not contain a monochromatic $P_{3n+1}^2$, we have a partition $X_1, X_2, Y_1, Y_2, Z, R$ of $V(G)$ with the conditions \ref{conditionA}-\ref{conditionG}, which we fix.

We now want to refine these conditions by adapting repeatedly a greedy procedure. Since we are going to apply multiple times the same method, we will explain the greedy procedure and the arguments for existence only in the first instance.

\begin{claim}\label{claim91}
$X_1$ and $X_2$ are entirely blue, while $Y_1$ and $Y_2$ are entirely red.
\end{claim}
\begin{claimproof}
Assume by contradiction that there is a red edge $x_1x_1'$ in $X_1$. Since $x_1$ and $x_1'$ have at most $\alpha n$ blue neighbours in $Z$ and $(1-\alpha)n\leq \vass{X}$, we have that $x_1$ and $x_1'$ have a common red neighbour $z\in Z$. Similarly, by considering the common red neighbour of $x_1$ and $x_1'$ in $Y_2$, we can find $y_2, y_2' \in Y_2$ such that $y_2y_2', x_1y_2, x_1'y_2', x_1'y_2, x_1y_2'$ are all red.

We are now ready to extend the red path $P_0=y_2, y_2', x_1, x_1', z$ (whose square is also monochromatic red) to a path $P$ of length bigger than $3n+2$ such that $P^2$ is also monochromatic red. The idea is to greedily add to $P_0$ at least $\frac{3}{2}n$ vertices from $Y_2$ (using the fact that almost all the edges in $Y_2$ are red) and $2n$ vertices from $(X_1, X_2, Z)$.

In order to do that, it suffices to show that we can find a path $P_{Y_2}$ of length at least $\frac{3}{2}n$ in $Y_2$ that starts with $y_2y_2'$ and such that $P_{Y_2}^2$ is monochromatic red. Assume we have built already a path $P_{Y_2}=y_2, y_2',\dots{}, p_{\ell}$ with the aforementioned conditions, provided $\ell<\frac{3}{2}n$, we can extend $P_{Y_2}$ simply by appending a vertex $p_{\ell+1}$ that is in the common red neighbour of $p_{\ell}$ and $p_{\ell-1}$ in $Y_2\setminus P_{Y_2}$. But this is possible, indeed all but at most $\frac{2}{1000}n$ vertices in $Y_2$ have red edges to both $p_{\ell}$ and $p_{\ell-1}$.

We do a very similar procedure by greedily extending $P_{(X_1, X_2, Z)}$. Given a path $P_{(X_1, X_2, Z)}=x_1', z, \dots{}, p_{\ell}$ of length smaller than $2n$, we can extend it by taking a vertex in the common red neighbour of $p_{\ell}$ and $p_{\ell-1}$ and in the right component.

Since $P_0^2$ is monochromatic red, and since we showed how to extend the endpoints to form a long path whose square is also monochromatic, we are done.

The arguments for $X_2, Y_1$ and $Y_2$ are symmetric.
\end{claimproof}

\begin{claim}\label{claim92}
The pairs $(X_1, Z)$ and $(X_2, Z)$ are entirely red, while the pairs $(Y_1, Z)$ and $(Y_2, Z)$ are entirely blue.
\end{claim}
\begin{claimproof}
Assume by contradiction that there is a blue edge $x_1z$ between $X_1$ and $Z$. Let $y_1\in Y_1$ be in the common blue neighbourhood of $x_1$ and $z$ (which exists by arguments similar to the ones above).

Take $x_1'\in X_1\setminus\llb x_1\rrb$ in the common blue neighbourhood of $y_1$ and $x_1$ and let $P_0=z,y_1,x_1,x_1'$. We have that $P_0^2$ is blue monochromatic. Also, we can greedily extend $P_0$ to a path $P$ such that $P^2$ is also blue monochromatic and $\vass{P}>3n$ by extending $x_1, x_1'$ to a path of length at least $\frac{3}{2}n$ in $X_1$ and extending the $zy_1$ end in $(Y_1, Y_2, Z)$ by at least $2n$ vertices.

The argument for the other pairs is symmetric.
\end{claimproof}

\begin{claim}\label{claim93}
The pairs $(X_1, Y_1)$ and $(X_2, Y_2)$ are entirely blue, while the pairs $(X_1, Y_2)$ and $(X_2, Y_1)$ are entirely red.
\end{claim}
\begin{claimproof}
Assume by contradiction that there is a red edge $x_1y_1$ in $(X_1, Y_1)$. The vertices $x_1$ and $y_1$ share a red neighbour $x_2$ in $X_2$. We can also find in $Y_1\setminus \llb y_1\rrb$ a common red neighbour $y_1'$ of $y_1$ and $x_2$. 

We can start with the path $P_0=x_1, x_2, y_1, y_1'$, and then extend it using vertices in $Y_1$ on one side and vertices of $(X_1, X_2, Z)$ on the other, until we get a path $P$ such that $P^2$ is monochromatic red and covers at least $3n+2$ vertices.

The argument for the other pairs is symmetric.
\end{claimproof}

\begin{claim}\label{claim94}
The pair $(X_1, X_2)$ has no blue $P_4$, while the pair $(Y_1, Y_2)$ has no red $P_4$.
\end{claim}
\begin{claimproof}
Assume $x_1x_2x_1'x_2'$ formed a blue $P_4$ in $(X_1, X_2)$. Since $X_1$ and $X_2$ are entirely blue, the edges $x_1x_1'$ and $x_2x_2'$ are blue. Each of these edges is the beginning of a square of a path covering the respective part. These join together to form a square of a path that is longer than allowed. 

The argument for the other pair is symmetric.
\end{claimproof}

From the claims above we can see that in the situation depicted by Lemma \ref{regularityresult} we have $\vass{X_1}, \vass{X_2}, \vass{Y_1}, \vass{Y_2}\leq 2n-1$. We can now partition the vertices of the remainder set $R$ depending on their neighbourhoods as follows.
\begin{itemize}
 \item[1)] Let us denote with $R_Z$ the set of vertices in $R$ with more than $\frac{n}{4}$ red neighbours both in $X_1$ and $X_2$,
 \item[2)] for $i=1,2$ let $R_i$ be the vertices in $R$ with more than $\frac{n}{4}$ blue neighbours in both $X_i$ and $Y_i$,
 \item[3)] let $R_{12}$ denote the vertices in $R$ with more than $\frac{n}{4}$ red neighbours in both $X_1$ and $Y_2$,
 \item[4)] let $R_{21}$ denote the vertices in $R$ with more than $\frac{n}{4}$ red neighbours in both $X_2$ and $Y_1$,
 \item[5)] let $R^*$ denote any vertices in $R$ that are not in any of the above sets.
\end{itemize}

\begin{claim}\label{claim95}
Vertices in $R^*$ have at least $\frac{3}{2}n$ blue neighbours in each $X_i$ and at least $\frac{3}{2}n$ red neighbours in each $Y_i$. Moreover, $\vass{R^*}\leq 1$.
\end{claim}
\begin{claimproof}
The first part of the claim is true by construction. Let us now assume that there are two vertices $u$ and $v$ in $R^*$. Then $u$ and $v$ have more than $\frac{n}{2}$ common blue neighbours in each $X_i$ and at least $\frac{n}{2}$ common red neighbours in each $Y_i$. Therefore if $uv$ is blue it will create a blue square of a path with vertices from $X_1$ and $X_2$, while if it is red it will join long red squares of paths in $Y_1$ and $Y_2$.
\end{claimproof}

\begin{claim}\label{cl:96}
We have the following bounds: $\vass{X_1\cup R_1}, \vass{X_2\cup R_2}, \vass{Y_1\cup R_{21}}, \vass{Y_2\cup R_{12}}\leq 2n-1$.
\end{claim}
\begin{claimproof}
Assume by contradiction that $\vass{X_1\cup R_1}\geq 2n$. Recall that in previous claims we proved that all the edges in $X_1$ and $(X_1, Y_1)$ are blue. Let us label the vertices in $R_1$ by $r_1,\dots{}, r_{\ell}$. Recall that $\ell=\vass{R_1}\leq \vass{R}\leq \alpha n$.

Since every vertex in $R_1$ has at least $\frac{n}{4}$ blue neighbours in both $X_1$ and $Y_1$ we can find disjoint blue triangles $T_1,\dots{}, T_{\ell}$ where triangle $T_i$ contains the vertices $r_i,x_i,y_i$ with $x_i\in X_1$ and $y_i\in Y_1$. We next find for each $i\in[\ell]$ vertices $a_i,b_i,c_i,a'_i,b'_i,c'_i$ as follows. We let $c_i$ be a blue neighbour of $r_i$ in $Y_1$, and $a_i,b_i\in X_1$, we let $a'_1$ be a neighbour in $X_1$ of $r_1$, $b'_1$ be in $X_1$, and $c'_1$ be in $Y_1$. Observe that since $\ell\le\alpha n$, we can ensure that all these vertices are different.

By construction, the vertex ordering $P_0=(a_1,b_1,c_1,x_1,r_1,y_1,a'_1,b'_1,c'_1,\dots)$, where we repeat the same letter ordering for $i=2$ and so on afterwards, is a blue square path. We extend $P_0$ further by choosing distinct vertices from $X_1$, $X_1$ and then $Y_1$ in this order, until no unused vertices remain in $X_1$. As $\vass{X_1\cup R_1}\geq 2n$, what we obtain is a blue square path with at least $3n$ vertices, if $\vass{X_1\cup R_1}\geq 2n+1$ we obtain at least $3n+1$ vertices. We can extend $P_0$ by one more vertex by adding a so far unused vertex of $Y_1$ at the start of the ordering. This gives the required $3n+1$-vertex square path (and $3n+2$ vertices if $\vass{X_1\cup R_1}\geq 2n+1$). The arguments for the other pairs of sets are the same.
\end{claimproof}

\begin{claim}\label{claim97}
We have that $\vass{Z\cup R_Z}\leq n-1$.
\end{claim}
\begin{claimproof}
Let as assume that $\vass{Z\cup R_Z}\geq n$ and let us label the vertices in $R_Z$ by $r_1,\dots{}, r_{\ell}$. Since $(X_1, X_2)$ has no blue path on $4$ vertices, there are at most $40$ vertices in $X_1\cup X_2$ with more than $\frac{n}{20}$ blue neighbours in the opposite part. Call the set of these vertices $X_{\text{bad}}$. Since each vertex in $R_Z$ has more than $\frac{n}{4}$ red neighbours in each $X_i$, we can find disjoint red triangles $T_1, \dots{}, T_{\ell}$ such that each $T_i$ uses $r_i$, a vertex $x_i^1\in X_{1}\setminus X_{\text{bad}}$ and a vertex $x_i^2\in X_2\setminus X_{\text{bad}}$. 

The idea is now to find for each $i\in [\ell]$ vertices $a_i, a_i'\in X_1$, $b_i, b_i'\in Z$, $c_i, c_i'\in X_2$ such that for every $i\in [\ell-1]$ we have that $(x_i^1, r_i, x_i^2, a_i, b_i, c_i, a_i', b_i', c_i', x_{i+1}^1, r_{i+1}, x_{i+1}^2)$ is a red square of a path. But this can be done greedily since $\ell\leq \alpha n$. We now build the path $P_0=(x_1^1, r_1, x_1^2, a_1, b_1, c_1, a_1', b_1', c_1', x_{2}^1, \dots{}, x_{\ell}^2)$ which by construction has the property that $P_0^2$ is red. 

We can extend $P_0$ by choosing distinct vertices from $X_1$, $Z$ and then $X_2$ in this order, until no unused vertices remain in $Z$. As $\vass{Z\cup R_Z}\geq n$, what we obtain is a red square path with at least $3n$ vertices.
\end{claimproof}

Putting the bounds from the last three claims together, we see $|G|\le 1+4(2n-1)+n-1=9n-4$, which contradiction completes the proof.
\end{proof}

The proof for $P_{3n+2}^2$ is almost verbatim as above (we actually worked with $P_{3n+2}^2$ in most of the claims), with the exception that in Claim~\ref{cl:96} we obtain the upper bound $|X_1\cup R_1|\le 2n$, as explained in the proof of that claim, and consequently a final upper bound $|G|\le 1+4(2n)+n-1=9n$ for a contradiction.

\begin{proof}[Sketch proof of cycle case of Theorem~\ref{thm:main}]
In order to prove that for $n$ large enough we have $R(C_{3n}^2,C_{3n}^2)=9n-3$, is suffices to modify our previous proof. We start by constructing the same partition we built at the beginning of the proof of Theorem \ref{thm:main} to get the sets $X_1, X_2, Y_1, Y_2, Z, R$. Now, by using the same technique introduced in Claim \ref{claim91} we can prove some weakened for of Claims \ref{claim91}, \ref{claim92}, \ref{claim93}, \ref{claim94}. Which is, we can prove that in $X_1$ we cannot find two disjoint red edges (the same holds for $X_2$), in $Y_1$ we cannot find two disjoint blue edges (the same holds for $Y_2$). Similarly, we cannot find two disjoint edges of the wrong colours in none of the following pairs: $(X_1, Z)$, $(X_2, Z)$, $(Y_1, Z)$, $(Y_2, Z)$, $(X_1, Y_1)$, $(X_2, Y_2), (X_1, Y_2), (X_2, Y_1)$. Moreover, we cannot find two vertex-disjoint $P_4$ of the wrong colour in $(X_1, X_2)$ nor in $(Y_1, Y_2)$.\\
To 
From these results and the previously proved Lemma \ref{regularityresult}, we can see that also in this case we have $\vass{X_1}, \vass{X_2}, \vass{Y_1}, \vass{Y_2}\leq 2n-1$. Now we can define the same partition of $R$ in sets $R_Z, R_1, R_2, R_{12}, R_{21}$ and $R^*$. 
Let us point out that from this modified version of Claim \ref{claim91} we have that there are two vertices $a, b\in X_1$ such that all edges in $G[X_1\setminus \llb a,b\rrb]$ are blue. In particular, from Claims \ref{claim91}, \ref{claim92}, \ref{claim93}, \ref{claim94} we get that up to moving at most $10$ vertices from $X_1$ to $R_1$ (and similarly from $X_2$ to $R_2$, from $Y_1$ to $R_{21}$, from $Y_2$ to $R_{12}$ and from $Z$ to $R_Z$) all the vertices in $X_1$ (and similarly in $X_2, Y_1, Y_2, Z$) are incident only to edges of the right colour in $G[X_1\cup X_2\cup Y_1\cup Y_2\cup Z]$, with the possible exception of edges in $(X_1, X_2)$ and $(Y_1, Y_2)$.\\
We now aim to explain how to modify Claim \ref{claim95} to hold for cycles and how to modify the proof of Claims \ref{cl:96}, \ref{claim97}. The first part of Claim \ref{claim95} holds by construction without any modifications. The second part of Claim \ref{claim95} needs to be modified to state that we cannot find two parallel edges of the same colour in $R^*$. Indeed, otherwise we could find a long monochromatic blue cycle $C$ such that $C^2$ is also blue by using vertices from $X_1$, $X_2$ and the two blue edges in $R^*$. In particular, this implies that $\vass{R^*}\leq 4$.
As a guide to show how to modify the proofs of Claim \ref{cl:96} and \ref{claim97}, we give a sketch of the modifications needed for Claim \ref{cl:96}. If we assume by contradiction that $\vass{X_1\cup R_1}\geq 2n$ we can almost verbatim repeat the same proof, having care of extending our path $P_0$ in both directions and making sure that the two endpoints of $P_0$ and their neighbours are adjacent in blue to each other. This is possible because $G[X_1]$ is entirely blue as claimed above. 
\begin{claim}
If $R^*$ contains a blue edge, then $\vass{X_1\cup R_1}, \vass{X_2\cup R_2}\leq 2n-2$ (same holds for red, $Y_1\cup R_{21}$ and $Y_2\cup R_{12}$).
\end{claim}
\begin{claimproof}
Assume $R^*$ contains a blue edge $uv$, then $\vass{X_1\cup R_1}\leq 2n-2$ (the arguments for the other cases are the same). In order to prove this, it suffices to show that there exists a maximal matching $T$ in $X_1$ such that we can build a blue cycle $C$ that covers all the edges of $X_1$, the two vertices $u, v\in R^*$ and, for each edge in $T$, an extra vertex in $Y_1$. This can be done because by Claim \ref{claim91} and Lemma \ref{regularityresult} there is a vertex $w\in X_1$ such that the red neighbourhood of $w$ in $X_1$ has size at most $\alpha n$, but $G[X_1\setminus w]$ has at most one red edge and because $u$ and $v$ have both at least $\frac{3}{2}n$ blue neighbours in $X_1$. Therefore, it is possible to build a cycle by replicating the construction in Claim \ref{cl:96} and by carefully adding the edge $uv$ to the cycle.
\end{claimproof}
This suffices to conclude. Indeed, if $\vass{R^*}\leq 3$ then we still have $\vass{X_1\cup R_1\cup R^*\cup X_2\cup R_2}\leq 4n-1$, while if $\vass{R^*}=4$ then we have both a red and a blue edge in $R^*$ (since we cannot have two vertex-disjoint edges of the same colour). In this case we have the following inequalities: $\vass{X_1\cup R_1}, \vass{X_2\cup R_2}, \vass{Y_1\cup R_{21}}, \vass{Y_2\cup R_{12}} \leq 2n-2$, which are enough to obtain the wanted bound.
\end{proof}



\printbibliography

@article{AllenBoettcherHladky2011,
    AUTHOR = {Allen, P. and B\"{o}ttcher, J. and Hladk\'{y}, J.},
     TITLE = {Filling the gap between {T}ur\'{a}n's theorem and {P}\'{o}sa's
              conjecture},
   JOURNAL = {J. Lond. Math. Soc. (2)},
  FJOURNAL = {Journal of the London Mathematical Society. Second Series},
    VOLUME = {84},
      YEAR = {2011},
    NUMBER = {2},
     PAGES = {269--302},
}

@incollection {SzeReg,
    AUTHOR = {Szemer\'{e}di, E.},
     TITLE = {Regular partitions of graphs},
 BOOKTITLE = {Probl\`emes combinatoires et th\'{e}orie des graphes ({C}olloq.
              {I}nternat. {CNRS}, {U}niv. {O}rsay, {O}rsay, 1976)},
    SERIES = {Colloq. Internat. CNRS},
    VOLUME = {260},
     PAGES = {399--401},
 PUBLISHER = {CNRS, Paris},
      YEAR = {1978},
}

@Unpublished{ABHKP,
   author = {{Allen}, P. and {B{\"o}ttcher}, J. and {H{\`a}n}, H. and {Kohayakawa}, Y. and 
	{Person}, Y.},
    title = "{Blow-up lemmas for sparse graphs}",
   note = {arXiv:1612.00622},
   year = {2016}
}

@article{CorradiHajnal1963,
    AUTHOR = {Corradi, K. and Hajnal, A.},
     TITLE = {On the maximal number of independent circuits in a graph},
   JOURNAL = {Acta Math. Acad. Sci. Hungar.},
  FJOURNAL = {Acta Mathematica. Academiae Scientiarum Hungaricae},
    VOLUME = {14},
      YEAR = {1963},
     PAGES = {423--439},
}

@article {KSS,
    AUTHOR = {Koml\'{o}s, J. and S\'{a}rk\"{o}zy, G. N. and Szemer\'{e}di, E.},
     TITLE = {Blow-up lemma},
   JOURNAL = {Combinatorica},
  FJOURNAL = {Combinatorica. An International Journal on Combinatorics and
              the Theory of Computing},
    VOLUME = {17},
      YEAR = {1997},
    NUMBER = {1},
     PAGES = {109--123},
}

@article {ABS,
    AUTHOR = {Allen, P. and Brightwell, G. and Skokan, J.},
     TITLE = {Ramsey-goodness---and otherwise},
   JOURNAL = {Combinatorica},
  FJOURNAL = {Combinatorica. An International Journal on Combinatorics and
              the Theory of Computing},
    VOLUME = {33},
      YEAR = {2013},
    NUMBER = {2},
     PAGES = {125--160},
}

@article {Burr,
    AUTHOR = {Burr, S. A.},
     TITLE = {Ramsey numbers involving graphs with long suspended paths},
   JOURNAL = {J. London Math. Soc. (2)},
  FJOURNAL = {Journal of the London Mathematical Society. Second Series},
    VOLUME = {24},
      YEAR = {1981},
    NUMBER = {3},
     PAGES = {405--413},
}

@article {GRR,
    AUTHOR = {Graham, R. L. and R\"{o}dl, V. and Ruci\'{n}ski, A.},
     TITLE = {On graphs with linear {R}amsey numbers},
   JOURNAL = {J. Graph Theory},
  FJOURNAL = {Journal of Graph Theory},
    VOLUME = {35},
      YEAR = {2000},
    NUMBER = {3},
     PAGES = {176--192},
}

@article {NikiRous,
    AUTHOR = {Nikiforov, V. and Rousseau, C. C.},
     TITLE = {Ramsey goodness and beyond},
   JOURNAL = {Combinatorica},
  FJOURNAL = {Combinatorica. An International Journal on Combinatorics and
              the Theory of Computing},
    VOLUME = {29},
      YEAR = {2009},
    NUMBER = {2},
     PAGES = {227--262},
}

@article {GG67,
    AUTHOR = {Gerencs\'{e}r, L. and Gy\'{a}rf\'{a}s, A.},
     TITLE = {On {R}amsey-type problems},
   JOURNAL = {Ann. Univ. Sci. Budapest. E\"{o}tv\"{o}s Sect. Math.},
  FJOURNAL = {Annales Universitatis Scientiarum Budapestinensis de Rolando
              E\"{o}tv\"{o}s Nominatae. Sectio Mathematica},
    VOLUME = {10},
      YEAR = {1967},
     PAGES = {167--170},
}

@article {R73,
    AUTHOR = {Rosta, V.},
     TITLE = {On a {R}amsey-type problem of {J}. {A}. {B}ondy and {P}.
              {E}rd\H{o}s. {I}, {II}},
   JOURNAL = {J. Combinatorial Theory Ser. B},
  FJOURNAL = {Journal of Combinatorial Theory. Series B},
    VOLUME = {15},
      YEAR = {1973},
     PAGES = {94--104; ibid. 15 (1973), 105--120},
}

@article {BE73,
    AUTHOR = {Bondy, J. A. and Erd\H{o}s, P.},
     TITLE = {Ramsey numbers for cycles in graphs},
   JOURNAL = {J. Combinatorial Theory Ser. B},
  FJOURNAL = {Journal of Combinatorial Theory. Series B},
    VOLUME = {14},
      YEAR = {1973},
     PAGES = {46--54},
}
\end{document}